\documentclass{amsart2020}
\newtheorem{theorem}{Theorem}[section]
\newtheorem{lemma}[theorem]{Lemma}
\newtheorem{proposition}[theorem]{Proposition}
\newtheorem{corollary}[theorem]{Corollary}

\theoremstyle{definition}
\newtheorem{definition}[theorem]{Definition}
\newtheorem{example}[theorem]{Example}
\newtheorem{remark}[theorem]{Remark}
\usepackage{amscd,amssymb}
\begin{document}

\title[Free bicommutative superalgebras]{Free bicommutative superalgebras}

\author[Drensky, Ismailov, Mustafa, Zhakhayev]{Vesselin Drensky, Nurlan Ismailov, Manat Mustafa, Bekzat Zhakhayev}
\address{Institute of Mathematics and Informatics, Bulgarian Academy of Sciences, Sofia 1113, Bulgaria}
\email{drensky@math.bas.bg}
\address{Astana IT University, Astana 010000, Kazakhstan and Suleyman Demirel University, Kaskelen, Kazakhstan}
\email{nurlan.ismail@gmail.com}
\address{Department of Mathematics, School of Sciences and Humanities, Nazarbayev University, 53 Qabanbay Batyr Avenue, Astana 010000, Kazakhstan}
\email{manat.mustafa@nu.edu.kz}
\address{Suleyman Demirel University and Institute of Mathematics and Mathematical Modeling, Almaty, Kazakhstan}
\email{bekzat.kopzhasar@gmail.com}



\thanks{This research was supported by the Science Committee of the Ministry of Science and Higher Education of the Republic of Kazakhstan (Grant No. AP09259551)
and by Nazarbayev University Faculty Development Competitive Research Grant 021220FD3851.}

\maketitle

\begin{abstract}
We introduce the variety ${\mathfrak B}_{\text{\rm sup}}$ of bicommutative superalgebras over an arbitrary field of characteristic different from 2.
The variety consists of all nonassociative ${\mathbb Z}_2$-graded algebras satisfying the polynomial super-identities of super- left- and right-commutativity
\[
x(yz)= (-1)^{\overline{x}\,\overline{y}} y(xz)\text{ and } (xy)z=(-1)^{\overline{y}\,\overline{z}} (xz)y,
\]
where $\overline{u}\in\{0,1\}$ is the parity of the homogeneous element $u$.

We present an explicit construction of the free bicommutative superalgebras, find their bases as vector spaces
and show that they share many properties typical for ordinary bicommutative algebras and super-commutative associative superalgebras.
In particular, in the case of free algebras of finite rank we compute the Hilbert series and find explicitly its coefficients.
As a consequence we give a formula for the codimension sequence. We establish an analogue of the classical Hilbert Basissatz for two-sided ideals.
We see that the Gr\"obner-Shirshov bases of these ideals are finite,
the Gelfand-Kirillov dimensions of finitely generated bicommutative superalgebras are nonnegative integers
and the Hilbert series of finitely generated graded bicommutative superalgebras are rational functions.
Concerning problems studied in the theory of varieties of algebraic systems, we prove that the variety of bicommutative superalgebras satisfies the Specht property.
In the case of characteristic 0 we compute the sequence of cocharacters.
\end{abstract}

\section{Introduction}

\subsection{Bicommutative algebras}
An algebra over a field $K$ belongs to the variety $\mathfrak B$ of {\it bicommutative} algebras if it satisfies the identities of
{\it left-commutativity} and {\it right-commutativity}, respectively,
\[
x_1(x_2x_3)=x_2(x_1x_3)\text{ and }(x_1x_2)x_3=(x_1x_3)x_2.
\]
Bicommutative algebras were first considered by Dzhumadil'daev and Tulenbaev in \cite{Dzhumadil'daev-Tulenbaev}.
Bicommutative algebras arise also in the study by Burde, Dekimpe, Deschamps and Vercammen of simply transitive affine actions on Lie groups, see \cite{Burde1}, \cite{Burde2}, \cite{Burde3}.
In the latter three papers, bicommutative algebras are called LR- (left and right) algebras.
Bases of free bicommutative algebras and some numerical invariants of bicommutative algebras were given by Dzhumadil'daev, Ismailov and Tulenbaev in \cite{Dzh-Ism-Tul}.
Polynomial identities satisfied by bicommutative algebras with respect to the commutator and anti-commutator products and Lie and Jordan elements
were studied by Dzhumadil'daev and Ismailov in \cite{Dzhumadil'daev-Ismailov2018}.
It was proved by Drensky and Zhakhayev in \cite{Drensky-Zhakhaev} that finitely generated bicommutative algebras are weakly noetherian
and the authors gave a positive solution of the Specht problem for varieties of bicommutative algebras over an arbitrary field of any characteristic.
The subvarieties of the variety $\mathfrak B$ of bicommutative algebras were studied by Drensky in \cite{Drensky2}.
In \cite{Shestakov-Zhang2021}, Shestakov and Zhang studied automorphisms of free bicommutative algebras over  an arbitrary field
and constructed wild automorphisms in two-generated and three-generated free bicommutative algebras.
Bai, Chen and Zhang showed in \cite{Bai-Chen-Zhang2022} that the Gelfand-Kirillov dimension of an arbitrary finitely generated bicommutative algebra is a nonnegative integer.
They also established in \cite{Bai-Chen-Zhang2022} that the ideals of finitely generated free bicommutative algebras $F_d({\mathfrak B})$ have a finite Gr\"obner-Shirshov bases.
The algebraic and geometric classifications of bicommutative algebras in low dimensions were studied
by Drensky \cite{Drensky2},  Kaygorodov and Volkov \cite{Kaygorodov-Volkov2019},  Kaygorodov, P\'aez-Guill\'an and Voronin \cite{Kaygorodov-Voronin2020}, \cite{Kaygorodov-Voronin2021}.

\subsection{Bicommutative superalgebras}
In this paper $K$ is an arbitrary field of characteristic different from 2. We introduce the class of {\it bicommutative superalgebras} over $K$.
The algebra $A$ is a {\it superalgebra} if it is $\mathbb{Z}_2$-graded,
i.e. $A=A_0+A_1$ is a direct sum of its subspaces $A_0$ and $A_1$ such that $A_iA_j\subseteq A_{i+j\,(\text{mod}\,2)}$.
The subspaces $A_0$ and $A_1$ are called the {\it even} and the {\it odd} components of $A$, respectively,
and their elements are called {\it even} and {\it odd}, respectively.
In what follows we always assume that the considered elements are homogeneous, i.e. they are either even or odd.
For an element $a\in A_i$, $i\in \{0,1\}$, we write $\overline{a}=i$ for its parity.
A superalgebra $A=A_0+A_1$  with super-identities
\begin{equation}\label{superleft}
x_1(x_2x_3)= (-1)^{\overline{x_1}\,\overline{x_2}} x_2(x_1x_3),
\end{equation}
\begin{equation}\label{superright}
(x_1x_2)x_3=(-1)^{\overline{x_2}\,\overline{x_3}} (x_1x_3)x_2
\end{equation}
is called a {\it bicommutative} superalgebra. The super-identity (\ref{superleft}) is called {\it super-left-commutative}
and the super-identity (\ref{superright}) is called {\it super-right-commutative}.
We denote by ${\mathfrak B}_{\text{\rm sup}}$ the variety of all bicommutative superalgebras.

For a variety $\mathfrak V$ of algebras one of the first problems to solve is to describe the free algebra $F({\mathfrak V})$ of countable rank
and the free algebra $F_d({\mathfrak V})$ of finite rank $d$.
For the description of $F({\mathfrak V})$ and $F_d({\mathfrak V})$ it is important to find an explicit basis as a vector space and a nice multiplication rule
which allows to study combinatorial aspects related with all algebras in $\mathfrak V$.
In the general case it is not always easy to solve these problems but for several well-known varieties of algebras the description of the free algebras is well understood.
Such examples include free associative algebras, Lie algebras, (anti)commutative algebras, and bicommutative algebras.
For a variety ${\mathfrak V}_{\text{\rm sup}}$ of superalgebras one has to describe the free algebra $F({\mathfrak V}_{\text{\rm sup}})$
generated by countable sets of even and odd variables and the algebra $F_{p,q}({\mathfrak V}_{\text{\rm sup}})$ generated by $p$ even and $q$ odd variables.
Such a description is given for example for free Lie superalgebras by Shtern \cite{Shtern1986}
and Bokut, Kang, Lee and Malcolmson \cite{Bokut1999}, free right-symmetric superalgebras by Vasilieva and Mikhalev \cite{Vasilieva-Mikhalev1996},
Malcev superalgebras on one odd generator by Shestakov \cite{Shestakov2003}, free alternative superalgebras on one odd generator by Shestakov and Zhukavets \cite{Shestakov-Zhukavets2007},
free Novikov (or Gelfand-Dorfman-Novikov) superalgebras by Zhang, Chen and Bokut \cite{Zhang-Chen-Bokut2019}.

We present an explicit construction of the free bicommutative superalgebras $F({\mathfrak B}_{\text{\rm sup}})$ and $F_{p,q}({\mathfrak B}_{\text{\rm sup}})$,
and find their bases as vector spaces. Comparing the situation with the ordinary bicommutative case, Dzhumadil'daev and Tulenbaev noted in \cite{Dzhumadil'daev-Tulenbaev},
that if $A$ is a bicommutative algebra, then $A^2$ is commutative and associative. In the super-case, we show that $A^2$ is an associative and super-commutative superalgebra.
We show that bicommutative superalgebras share many properties typical for ordinary bicommutative algebras and super-commutative associative superalgebras
in the spirit of the results of Drensky and Zhakhayev in \cite{Drensky-Zhakhaev}.
In particular, in the case of free algebras of finite rank we compute the Hilbert series and find explicitly their coefficients.
As a consequence we give a formula for the codimension sequence.
We establish an analogue of the classical Hilbert Basissatz for two-sided ideals.
We see that the Gr\"obner-Shirshov bases of the ideals in $F_{p,q}({\mathfrak B}_{\text{\rm sup}})$ are finite.
As a consequence we obtain that the Gelfand-Kirillov dimension of a finitely generated bicommutative superalgebra is a nonnegative integer
and the Hilbert series of finitely generated graded bicommutative superalgebras are rational functions.
This is an analogue of results by Bai, Chan and Zhang in \cite{Bai-Chen-Zhang2022} but the methods are in the spirit of \cite{Drensky-Zhakhaev}.

Then we consider problems typical for the theory of varieties of algebraic systems.
Using the Higman-Cohen method \cite{Higman} and \cite{Cohen} we prove
that the variety ${\mathfrak B}_{\text{\rm sup}}$ of bicommutative superalgebras over an arbitrary field $K$ satisfies the Specht property
which is a super-analogue of a result in \cite{Drensky-Zhakhaev}.
This is a consequence of a stronger result for Gr\"obner-Shirshov bases of ideals invariant under semigroups of endomorphisms
in the spirit of the results by Drensky and La Scala \cite{Drensky-La Scala}.
In the case of characteristic 0 we compute the sequences of cocharacters as in the paper by Dzhumadil'daev, Ismailov and Tulenbaev \cite{Dzh-Ism-Tul}
but again the proofs are as in \cite{Drensky-Zhakhaev}.

\section{Bases of free bicommutative superalgebras}

\subsection{Free bicommutative algebras}
Recall that the following basis of the free bicommutative algebra $F({\mathfrak B})$ was constructed in
\cite{Dzh-Ism-Tul}. Let $F({\mathfrak B})$ be freely generated by the set $X$ with a well-ordering $\leq$.
Suppose that $x_1,\ldots,x_k,y_1,\ldots,y_l\in X$ with $x_1\leq\cdots\leq x_k$ and $y_1\leq\cdots\leq y_l$ with $k,l\geq 1$.
Then set of elements of the form
\begin{equation}\label{basis of algebra}
x_1(\cdots(x_{k-1}((\cdots((x_ky_1)y_2)\cdots)y_l))\cdots)
\end{equation}
together with the free generating set $X$ forms a basis of $F({\mathfrak B})$.
It was shown in \cite{Drensky-Zhakhaev} that the algebra $F({\mathfrak B})$ is isomorphic to the following bicommutative algebra.
Let $K[X]$ be the usual polynomial algebra and let $\omega(K[X])$ be the augmentation ideal of $K[X]$
(the set of polynomials without constant term).
The algebra $G(X)$ is a direct sum of the vector spaces $KX$ with basis $X$ and $\omega(K[X])\otimes_K\omega(K[X])$ with multiplication rule
\[
(KX+\omega(K[X])\otimes\omega(K[X]))\circ (KX+\omega(K[X])\otimes\omega(K[X]))\to\omega(K[X])\otimes\omega(K[X])
\]
defined by
\[
x\circ y=x\otimes y,\quad x,y\in X,
\]
\[
x\circ(u(X)\otimes v(X))=(xu(X))\otimes
v(X),(u(X)\otimes v(X))\circ x=u(X)\otimes(v(X)x),
\]
$x\in X$, $u(X),v(X)\in\omega(K[X])$,
\[
(u_1(X)\otimes v_1(X))\circ(u_2(X)\otimes v_2(X))=(u_1(X)u_2(X))\otimes(v_1(X)v_2(X)),
\]
$u_1(X),u_2(X),v_1(X),v_2(X)\in\omega(K[X])$. The isomorphism $F({\mathfrak B})\cong G(X)$ is defined by
\[
x\to x\text{ for }x\in X, \] \[ x_1(\cdots(x_{k-1}((\cdots((x_ky_1)y_2)\cdots)y_l))\cdots)\to x_1\cdots
x_{k-1}x_k\otimes y_1y_2\cdots y_l
\]
for $x_1(\cdots(x_{k-1}((\cdots((x_ky_1)y_2)\cdots)y_l))\in F^2({\mathfrak B})$.

\subsection{Prebases of free bicommutative superalgebras}
We shall find similar constructions in the super-case. In what follows we shall work in the free bicommutative superalgebra $F({\mathfrak B}_{\text{\rm sup}})$
with a set of free generators $X=Y\cup Z$, where $Y=X_0$ and $Z=X_1$ are the well-ordered sets of even and odd generators, respectively.
Define a set $M$ consisting of monomials of the form
\begin{equation}\label{basis of superalgebra}
y_1(\cdots(y_k(z_1(\cdots(z_{l-1}((\cdots(((\cdots(z_ly_{k+1})\cdots)y_{k+m})z_{l+1})\cdots)z_{l+n}))\cdots)))\cdots)
\end{equation}
with
\[
 y_1\leq\cdots\leq y_k, \, y_{k+1}\leq\cdots\leq y_{k+m},\, z_1<\cdots<z_l, \,z_{l+1}<\cdots<z_{l+n},
\]
where $y_1,\ldots, y_{k+m}\in Y$, $z_1,\ldots, z_{l+n}\in Z$ and $k+l>0,  m+n>0$.
We shall show that the set $X\cup M$ forms a basis of the vector space $F({\mathfrak B}_{\text{\rm sup}})$
and shall find rules for the multiplication of the elements in $X$ and $M$.

First, in a series of lemmas we shall show that the set $X\cup M$ is a pre-basis for $F({\mathfrak B}_{\text{\rm sup}})$,
i.e. $X\cup M$ is the spanning set for the free bicommutative superalgebra $F({\mathfrak B}_{\text{\rm sup}})$.

\begin{lemma}\label{Lemma 2.1}
The variety ${\mathfrak B}_{\text{\rm sup}}$ satisfies the super-identities
\begin{equation}\label{weak-associativity}
x_1((x_2x_3)x_4)=(x_1(x_2x_3))x_4;
\end{equation}
\begin{equation}\label{two odd in left-box}
x_2((\cdots((x_1x_3)x_4)\cdots)x_k)=(-1)^{\overline{x_1}\,\overline{x_2}}x_1((\cdots((x_2x_3)x_4)\cdots)x_k);
\end{equation}
\begin{equation}\label{cor two odd in left-box}
z((\cdots((zx_1)x_2)\cdots)x_k)=0,\, z\in Z;
\end{equation}
\begin{equation}\label{two odd in left}
\begin{split}
&x_2(x_k(\cdots(x_4(x_1x_3))\cdots))\\
&=(-1)^{(\overline{x_4}+\cdots+\overline{x_k})(\overline{x_1}+\overline{x_2})+\overline{x_1}\,\overline{x_2}}x_1(x_k(\cdots(x_4(x_2x_3))\cdots));
\end{split}
\end{equation}
\begin{equation}\label{cor two odd in left}
z(x_k(\cdots(x_2(zx_1))\cdots))=0,\, z\in Z;
\end{equation}
\begin{equation}\label{two odd in right}
\begin{split}
&((\cdots((x_3x_1)x_4)\cdots)x_k)x_2\\
&=(-1)^{(\overline{x_4}+\cdots+\overline{x_k})(\overline{x_1}+\overline{x_2})+\overline{x_1}\,\overline{x_2}}((\cdots((x_3x_2)x_4)\cdots)x_k)x_1;
\end{split}
\end{equation}
\begin{equation}\label{cor two odd in right}
((\cdots((x_1z)x_2)\cdots)x_k)z=0,\,z\in Z.
\end{equation}
\end{lemma}

\begin{proof}
By the super-identities (\ref{superleft}) and (\ref{superright})  we have
\[
x_1((x_2x_3)x_4)=(-1)^{\overline{x_3}\,\overline{x_4}}x_1((x_2x_4)x_3)=(-1)^{\overline{x_3}\,\overline{x_4}+\overline{x_1}\,\overline{x_2 x_4}}(x_2x_4)(x_1x_3)
\]
\[
=(-1)^{\overline{x_3}\,\overline{x_4}+\overline{x_1}\,\overline{x_2x_4}+\overline{x_4}\,\overline{x_1x_3}}(x_2(x_1x_3))x_4
\]
\[
=(-1)^{\overline{x_3}\,\overline{x_4}+\overline{x_1}\,\overline{x_2
x_4}+\overline{x_4}\,\overline{x_1 x_3}+\overline{x_1}\,\overline{x_2}}(x_1(x_2x_3))x_4
\]
\[
=(-1)^{\overline{x_3}\,\overline{x_4}+\overline{x_1}\,\overline{x_2}+\overline{x_1}\,
\overline{x_4}+\overline{x_1}\,\overline{x_4}+\overline{x_3}\,\overline{x_4}+\overline{x_1}\,\overline{x_2}}(x_1(x_2x_3))x_4
=(x_1(x_2x_3))x_4
\]
and this proves (\ref{weak-associativity}).
\[
x_{k+2}((\cdots((x_{k+1}x_1)x_2)\cdots)x_k)
\]
(by (\ref{weak-associativity}))
\[
=(\cdots((x_{k+2}(x_{k+1}x_1))x_2)\cdots)x_k
\]
(by (\ref{superleft}))
\[
=(-1)^{\overline{x_{k+1}}\,\overline{x_{k+2}}}(\cdots((x_{k+1}(x_{k+2}x_1))x_2)\cdots)x_k
\]
(by (\ref{weak-associativity}))
\[
=(-1)^{\overline{x_{k+1}}\,\overline{x_{k+2}}}x_{k+1}((\cdots((x_{k+2}x_1)x_2)\cdots)x_k)
\]
and this is (\ref{two odd in left-box}).

The super-identity (\ref{cor two odd in left-box}) follows immediately from (\ref{two odd in left-box}).

For the proof of (\ref{two odd in left}) we use the super-left-commutativity (\ref{superleft}):
\[
x_2(x_k(\cdots(x_4(x_1x_3))\cdots))
=(-1)^{(\overline{x_4}+\cdots+\overline{x_k})\overline{x_2}}x_k(\cdots(x_4(x_2(x_1x_3)))\cdots)
\]
\[
=(-1)^{(\overline{x_4}+\cdots+\overline{x_k})\overline{x_2}+\overline{x_1}\,\overline{x_2}}x_k(\cdots(x_4(x_1(x_2x_3)))\cdots)
\]
\[
(-1)^{(\overline{x_4}+\cdots+\overline{x_k})(\overline{x_1}+\overline{x_2})+\overline{x_1}\,\overline{x_2}}x_1(x_k(\cdots(x_4(x_2x_3))\cdots)).
\]

The super-identity (\ref{cor two odd in left}) is an immediate consequence of (\ref{two odd in left}).

For the proof of (\ref{two odd in right}) we use the super-right-commutativity (\ref{superright}) and obtain consecutively
\[
((\cdots((x_3x_1)x_4)\cdots)x_k)x_2
=(-1)^{(\overline{x_4}+\cdots+\overline{x_k})\overline{x_2}}(\cdots(((x_3x_1)x_2)x_4)\cdots)x_k
\]
\[
=(-1)^{(\overline{x_4}+\cdots+\overline{x_k})\overline{x_2}+\overline{x_1}\,\overline{x_2}}(\cdots(((x_3x_2)x_1)x_4)\cdots)x_k
\]
\[
=(-1)^{(\overline{x_4}+\cdots+\overline{x_k})(\overline{x_1}+\overline{x_2})+\overline{x_1}\,\overline{x_2}}((\cdots((x_3x_2)x_4)\cdots)x_k)x_1.
\]

The equation (\ref{cor two odd in right}) is an immediate consequence of (\ref{two odd in right}).
\end{proof}

\begin{corollary}\label{permuting the variables}
Let $\sigma$ be a permutation in the symmetric group $S_{k+l}$,
and let $x_1=y_1,\ldots,x_k=y_k\in Y$, $x_{k+1}=z_1,\ldots,x_{k+l}=z_l\in Z$, $u\in F({\mathfrak B}_{\text{\rm sup}})$. Then
\[
x_{\sigma(1)}(\cdots(x_{\sigma(k+l)}u)\cdots)=(-1)^{\sigma_1}y_1(\cdots(y_k(z_1(\cdots(z_{l-1}(z_lu))\cdots)))\cdots),
\]
\[
(\cdots(ux_{\sigma(1)})\cdots)x_{\sigma(k+l)}=(-1)^{\sigma_1}(\cdots(((\cdots((uy_1)y_2)\cdots)y_k)z_1)\cdots)z_l,
\]
where $(-1)^{\sigma_1}$ is the sign of the part $(\sigma(k+1),\ldots,\sigma(k+l))$ of the permutation $\sigma$
which shows the number of inversions between the odd variables $z_1,\ldots,z_l$.
\end{corollary}

\begin{proof}
The super-identities in the statement of the corollary follow from  the super-identities (\ref{two odd in left-box}), (\ref{two odd in left}), (\ref{two odd in right}),
(\ref{cor two odd in left-box}), (\ref{cor two odd in left}) and (\ref{cor two odd in right}) established in Lemma \ref{Lemma 2.1}.
\end{proof}

\begin{proposition}\label{pre-base}
Over a field $K$ of characteristics different from two the set $X\cup M$ is a pre-basis for $F({\mathfrak B}_{\text{\rm sup}})$.
\end{proposition}

\begin{proof}
Repeating the arguments in \cite{Dzh-Ism-Tul}, we obtain that the position of the parentheses in the monomials in $M$ is the same
as the position of the parentheses in the basis monomials (\ref{basis of algebra}) of free bicommutative algebras.
The ordering among even generators and odd generators can be arranged applying Corollary \ref{permuting the variables}.
\end{proof}

The next step in our considerations is to find the multiplication rules of the pre-basis monomials in $F({\mathfrak B}_{\text{\rm sup}})$.

\begin{lemma}\label{lem2}
If $u=(\cdots (u_1 u_2)\cdots )u_k$ and $v\in F^2({\mathfrak B}_{\text{\rm sup}})$, then
\[
uv= (-1)^{(\overline{u_2}+\cdots +\overline{u_k})\overline{v}}
u_1((\cdots (v u_2)\cdots )u_k).
\]
\end{lemma}

\begin{proof}
\[
uv= ((\cdots (u_1 u_2)\cdots )u_k)v
\]
(by (\ref{superright}))
\[
=(-1)^{\overline{v}(\overline{u_2}+\cdots+\overline{u_k})} (\cdots((u_1v)u_2)\cdots )u_k
\]
(using $k-1$ times (\ref{weak-associativity}))
\[
=(-1)^{\overline{v}(\overline{u_2}+\overline{u_3}+\cdots+\overline{u_k})} u_1((\cdots ((v u_2)u_3)\cdots)u_k).
\]
\end{proof}

\begin{lemma}\label{main lemma}
Let $k, l, m, n\geq 1$. Then for all $t_i,u_i,v_i,w_i\in X$
\[
(t_k(\cdots ((\cdots (t_1 u_1)\cdots )u_l)\cdots)) (v_m(\cdots  ((\cdots (v_1 w_1)\cdots )w_n)\cdots))
\]
\[
=(-1)^{\overline{u}\,(\overline{v}+\overline{w})}
t_k (\cdots(t_1(v_m (\cdots ((\cdots (((\cdots (v_1 w_1)\cdots )w_n)u_1)\cdots)u_l)\cdots))) \cdots),
\]
where $\overline{u}=\overline{u_1}+\cdots+\overline{u_l}$,
$\overline{v}=\overline{v_1}+\cdots+\overline{v_m}$ and  $\overline{w}=\overline{w_1}+\cdots+\overline{w_n}$.
\end{lemma}

\begin{proof}
Set
\[
u_0= (\cdots(t_1 u_1)\cdots)u_l \quad \text{and}\quad w_0=  (\cdots (v_1 w_1)\cdots)w_n.
\]
Then
\[
(t_k(\cdots ((\cdots (t_1 u_1)\cdots )u_l)\cdots)) (v_m(\cdots  ((\cdots (v_1 w_1)\cdots )w_n)\cdots)) \] \[ = (t_k (\cdots (t_2 u_0)\cdots ))\,(v_m (\cdots (v_2 w_0)\cdots ))
\]
(using $k-1$ times (\ref{weak-associativity}))
\[
=t_k(\cdots(t_2(u_0 (v_m(\cdots (v_2 w_0)\cdots ))))\cdots)
\]
(by (\ref{superleft}))
\[
=(-1)^{(\overline{v}-\overline{v_1})\overline{u_0}}  t_k(\cdots(t_2(v_m (\cdots (v_2 (u_0 w_0))\cdots)))\cdots)
\]
(by Lemma \ref{lem2})
\[
=(-1)^{(\overline{v}-\overline{v_1})\overline{u_0}+\overline{u}\,\overline{w_0}}t_k(\cdots(t_2(v_m (\cdots (v_2 (t_1 ((\cdots (w_0 u_1)\cdots)u_l)))\cdots )))\cdots)
\]
(by (\ref{superleft}))
\[
=(-1)^{(\overline{v}-\overline{v_1})(\overline{u_0}+\overline{t_1})+\overline{u}\,\overline{w_0}}t_k(\cdots(t_1(v_m (\cdots (v_2 ((\cdots (w_0 u_1)\cdots)u_l))\cdots )))\cdots)
\]
\[
=(-1)^{(\overline{v}-\overline{v_1})\,\overline{u}+\overline{u}(\overline{v_1}+\overline{w})}t_k(\cdots(a_1(v_m (\cdots (v_2 ((\cdots (w_0 u_1)\cdots)u_l))\cdots )))\cdots)
\]
\[
=(-1)^{\overline{u}\,(\overline{v}+\overline{w})} t_k (\cdots(t_1(v_m (\cdots ((\cdots (((\cdots (v_1 w_1)\cdots )w_n)u_1)\cdots)u_l)\cdots))) \cdots).
\]
\end{proof}

\subsection{The superalgebra $G_{\text{\rm sup}}(X)$}
Recall that a ${\mathbb Z}_2$-graded algebra is {\it super-commutative}, see e.g. \cite[page 76]{Varadarajan}, if it satisfies the super-identity
\[
x_2x_1=(-1)^{\overline{x_1}\,\overline{x_2}}x_1x_2.
\]
Let $F({\mathfrak A}_{\text{\rm sup}})$ be the free super-commutative algebra freely generated by $X=Y\cup Z$, where the sets $Y=X_0$ and $Z=X_1$ are well-ordered.
Then $F({\mathfrak A}_{\text{\rm sup}})$ has defining relations
\[
yx=xy,\,x\in X,y\in Y,\quad z_2z_1=-z_1z_2,\,z_1,z_2\in Z,
\]
and a basis consisting of all monomials
\[
y_1\dots y_kz_1\cdots z_l,\,y_1\leq\cdots\leq y_k,z_1<\cdots<z_l.
\]
Let $\omega(F({\mathfrak A}_{\text{\rm sup}}))$ be the augmentation ideal of $F({\mathfrak A}_{\text{\rm sup}})$.
As in \cite{Drensky-Zhakhaev} we define the algebra
\[
G_{\text{\rm sup}}(X)=KX+\omega(F({\mathfrak A}_{\text{\rm sup}}))\otimes_K\omega(F({\mathfrak A}_{\text{\rm sup}}))
\]
with multiplication rules
\[
x_i\circ x_j=x_i\otimes x_j,\, x_i,x_j\in X,
\]
\[
x\circ(u\otimes v)=(xu)\otimes v,\,(u\otimes v)\circ x=u\otimes (vx),\, x\in X, u,v\in\omega(F({\mathfrak A}_{\text{\rm sup}})),
\]
\[
(u_1\otimes v_1)\circ(u_2\otimes v_2)=(-1)^{\overline{v_1}(\overline{u_2}+\overline{v_2})}(u_1u_2)\otimes(v_2v_1), \,u_1,v_1,u_2,v_2\in\omega(F({\mathfrak A}_{\text{\rm sup}})).
\]
Here the parity of the elements of $G^2_{\text{\rm sup}}(X)$ is determined by the rule
\[
\overline{u\otimes v}\equiv\overline{u}+\overline{v}\,\text{(mod 2)}.
\]

\begin{lemma}\label{the algebra G}
{\rm (i)} The superalgebra $G_{\text{\rm sup}}(X)$ belongs to the variety ${\mathfrak B}_{\text{\rm sup}}$.

{\rm (ii)} Its square $G_{\text{\rm sup}}^2(X)$ is an associative super-commutative superalgebra.
\end{lemma}

\begin{proof}
(i) Let $x_i\in X$, $u_i\otimes v_i\in G^2(X_\text{\rm sup})$, $i=1,2,3$. Then
\[
x_1\circ(x_2\circ x_3)=x_1\circ(x_2\otimes x_3)=(x_1x_2)\otimes x_3
\]
\[
=(-1)^{\overline{x_1}\,\overline{x_2}}(x_2x_1)\otimes x_3=(-1)^{\overline{x_1}\,\overline{x_2}}x_2\circ(x_1\circ x_3);
\]
\[
x_1\circ(x_2\circ(u_3\otimes v_3))=(x_1x_2u_3)\otimes v_3
\]
\[
=(-1)^{\overline{x_1}\,\overline{x_2}}(x_2x_1u_3)\otimes v_3=(-1)^{\overline{x_1}\,\overline{x_2}}x_2\circ(x_1\circ(u_3\otimes v_3)).
\]
With similar considerations we establish
\[
(x_1\circ x_2)\circ x_3=(-1)^{\overline{x_2}\,\overline{x_3}}(x_1\circ x_3)\circ x_2;
\]
\[
((u_1\otimes v_1)\circ x_2)\circ x_3=(-1)^{\overline{x_2}\,\overline{x_3}}((u_1\otimes v_1)\circ x_3)\circ x_2.
\]
Further we obtain
\[
(u_1\otimes v_1)\circ(x_2\circ x_3)=(-1)^{\overline{v_1}(\overline{x_2}+\overline{x_3})}(u_1x_2)\otimes(x_3v_1)
\]
\[
=(-1)^{\overline{v_1}(\overline{x_2}+\overline{x_3})}((-1)^{\overline{x_2}\,\overline{u_1}}(x_2u_1)\otimes(-1)^{\overline{x_3}\,\overline{v_1}}(v_1x_3))
\]
\[
=(-1)^{\overline{x_2}(\overline{u_1}+\overline{v_1})}x_2\circ((u_1\otimes v_1)\circ x_3);
\]
\[
(x_1\circ x_2)\circ(u_3\otimes v_3)=(-1)^{\overline{x_2}(\overline{u_3}+\overline{v_3})}(x_1u_3)\otimes (v_3x_2)
\]
\[
=(-1)^{\overline{x_2}(\overline{u_3}+\overline{v_3})}(x_1\circ(u_3\otimes v_3))\circ x_2;
\]
\[
(u_1\otimes v_1)\circ(x_2\circ(u_3\otimes v_3))=(-1)^{\overline{v_1}(\overline{x_2}+\overline{u_3}+\overline{v_3})}(u_1x_2u_3)\otimes(v_3v_1)
\]
\[
=(-1)^{\overline{v_1}(\overline{x_2}+\overline{u_3}+\overline{v_3})}((-1)^{\overline{x_2}\,\overline{u_1}}(x_2u_1u_3)\otimes(v_3v_1))
\]
\[
=(-1)^{\overline{x_2}(\overline{u_1}+\overline{v_1})}x_2\circ((-1)^{\overline{v_1}(\overline{u_3}+\overline{v_3})}(u_1u_3)\otimes(v_3v_1))
\]
\[
=(-1)^{\overline{x_2}(\overline{u_1}+\overline{v_1})}x_2\circ((u_1\otimes v_1)\circ(u_3\otimes v_3));
\]
\[
((u_1\otimes v_1)\circ x_2)\circ(u_3\otimes v_3)=(-1)^{(\overline{v_1}+\overline{x_2})(\overline{u_3}+\overline{v_3})}(u_1u_3)\otimes(v_3v_1x_2)
\]
\[
=(-1)^{\overline{x_2}(\overline{u_3}+\overline{v_3})}((-1)^{\overline{v_1}(\overline{u_3}+\overline{v_3})}(u_1u_3)\otimes(v_3v_1))\circ x_2
\]
\[
=(-1)^{\overline{x_2}(\overline{u_3}+\overline{v_3})}((u_1\otimes v_1)\circ(u_3\otimes v_3))\circ x_2;
\]
\[
(u_1\otimes v_1)\circ((u_2\otimes v_2)\circ x_3)=(-1)^{\overline{v_1}(\overline{u_2}+\overline{v_2}+\overline{x_3})}(u_1u_2)\otimes(v_2x_3v_1)
\]
\[
=(-1)^{\overline{v_1}(\overline{u_2}+\overline{v_2}+\overline{x_3})}((-1)^{\overline{u_1}\,\overline{u_2}}u_2u_1)
\otimes((-1)^{\overline{x_3}(\overline{v_1}+\overline{v_2})+\overline{v_1}\,\overline{v_2}}v_1x_3v_2)
\]
\[
=(-1)^{\overline{u_2}(\overline{u_1}+\overline{v_1})+\overline{v_2}\,\overline{x_3}}(u_2u_1)\otimes((v_1x_3)v_2)
\]
\[
=(-1)^{\overline{u_2}(\overline{u_1}+\overline{v_1})+\overline{v_2}\,\overline{x_3}}(-1)^{\overline{v_2}(\overline{u_1}+\overline{v_1}+\overline{x_3})}(u_2\otimes v_2)\circ(u_1\otimes(v_1x_3))
\]
\[
=(-1)^{(\overline{u_1}+\overline{v_1})(\overline{u_2}+\overline{v_2})}(u_2\otimes v_2)\circ((u_1\otimes v_1)\circ x_3);
\]
\[ (x_1\circ(u_2\otimes v_2))\circ(u_3\otimes v_3)=(-1)^{\overline{v_2}(\overline{u_3}+\overline{v_3})}(x_1u_2u_3)\otimes(v_3v_2)
\]
\[
=(-1)^{\overline{v_2}(\overline{u_3}+\overline{v_3})}(-1)^{\overline{u_2}\,\overline{u_3}+\overline{v_2}\,\overline{v_3}}(x_1u_3u_2)\otimes(v_2v_3)
=(-1)^{(\overline{u_2}+\overline{v_2})\overline{u_3}}(x_1u_3u_2)\otimes(v_2v_3)
\]
\[
=(-1)^{(\overline{u_2}+\overline{v_2})\overline{u_3}}(-1)^{(\overline{u_2}+\overline{v_2})\overline{v_3}}((x_1u_3)\otimes v_3)\circ(u_2\otimes v_2)
\]
\[
=(-1)^{(\overline{u_2}+\overline{v_2})(\overline{u_3}+\overline{v_3})}(x_1\circ(u_3\otimes v_3))\circ(u_2\otimes v_2);
\]
\[
(u_1\otimes v_1)\circ((u_2\otimes v_2)\circ(u_3\otimes v_3))=(u_1\otimes v_1)\circ(-1)^{\overline{v_2}(\overline{u_3}+\overline{v_3})}((u_2u_3)\otimes(v_3v_2))
\]
\[
=(-1)^{\overline{v_1}(\overline{u_2}+\overline{u_3}+\overline{v_3}+\overline{v_2})}(-1)^{\overline{v_2}(\overline{u_3}+\overline{v_3})}(u_1u_2u_3)\otimes(v_3v_2v_1)
\]
\[
(-1)^{\overline{v_1}(\overline{u_2}+\overline{u_3}+\overline{v_3})+\overline{v_2}(\overline{u_3}+\overline{v_3})+\overline{u_1}\,\overline{u_2}}(u_2u_1u_3)\otimes(v_3v_1v_2)
\]
\[
=(-1)^{\overline{v_1}(\overline{u_2}+\overline{u_3}+\overline{v_3})+\overline{v_2}(\overline{u_3}+\overline{v_3})+\overline{u_1}\,\overline{u_2}}
(-1)^{\overline{v_2}(\overline{u_1}+\overline{u_3}+\overline{v_3}+\overline{v_1})}(u_2\otimes v_2)\circ((u_1u_3)\otimes(v_3v_1))
\]
\[
=(-1)^{(\overline{u_1}+\overline{v_1})(\overline{u_2}+\overline{v_2})}(u_2\otimes v_2)\circ(-1)^{\overline{v_1}(\overline{u_3}+\overline{v_3})}(u_1u_3)\otimes(v_3v_1)
\]
\[
=(-1)^{(\overline{u_1}+\overline{v_1})(\overline{u_2}+\overline{v_2})}(u_2\otimes v_2)\circ((u_1\otimes v_1)\circ(u_3\otimes v_3));
\]
\[
((u_1\otimes v_1)\circ(u_2\otimes v_2))\circ(u_3\otimes v_3)=(-1)^{\overline{v_1}(\overline{u_2}+\overline{v_2})}((u_1u_2)\otimes(v_2v_1))\circ(u_3\otimes v_3)
\]
\[
=(-1)^{\overline{v_1}(\overline{u_2}+\overline{v_2})}(-1)^{(\overline{v_1}+\overline{v_2})(\overline{u_3}+\overline{v_3})}(u_1u_2u_3)\otimes(v_3v_2v_1)
\]
\[
=(-1)^{\overline{v_1}(\overline{u_2}+\overline{v_2}+\overline{u_3}+\overline{v_3})+(\overline{u_2}+\overline{v_2})\overline{u_3}}((u_1u_3)u_2)\otimes(v_2(v_3v_1))
\]
\[
=(-1)^{\overline{v_1}(\overline{u_2}+\overline{v_2}+\overline{u_3}+\overline{v_3})+(\overline{u_2}+\overline{v_2})\overline{u_3}}(-1)^{(\overline{v_3}+\overline{v_1})(\overline{u_2}+\overline{v_2})}
((u_1u_3)\otimes(v_3v_1))\circ(u_2\otimes v_2) \] \[ =(-1)^{(\overline{v_1}+\overline{u_2}+\overline{v_2})(\overline{u_3}+\overline{v_3})}((u_1u_3)\otimes(v_3v_1))\circ(u_2\otimes v_2)
\]
\[
=(-1)^{(\overline{u_2}+\overline{v_2})(\overline{u_3}+\overline{v_3})}((u_1\otimes v_1)\circ(u_3\otimes v_3))\circ(u_2\otimes v_2).
\]

(ii) For $u_i\otimes v_i\in G^2(X_\text{\rm sup})$, $i=1,2,3$, we obtain
\[
((u_1\otimes
v_1)\circ(u_2\otimes v_2))\circ(u_3\otimes v_3) =(-1)^{\overline{v_1}(\overline{u_2}+\overline{v_2}+\overline{u_3}+\overline{v_3})+\overline{v_2}(\overline{u_3}+\overline{v_3})}(u_1u_2u_3)\otimes(v_3v_2v_1)
\]
\[
=(u_1\otimes v_1)\circ((u_2\otimes v_2)\circ(u_3\otimes v_3))
\]
which shows the associativity of $G^2(X_\text{\rm sup})$.

For $u_i\otimes v_i\in G^2(X_\text{\rm sup})$, $i=1,2$, we compute consecutively
\[
(u_1\otimes v_1)\circ(u_2\otimes v_2)-(-1)^{(\overline{u_1}+\overline{v_1})(\overline{u_2}+\overline{v_2})}(u_2\otimes v_2)\circ(u_1\otimes v_1)
\]
\[
=(-1)^{\overline{v_1}(\overline{u_2}+\overline{v_2})}(u_1u_2)\otimes(v_2v_1)
-(-1)^{(\overline{u_1}+\overline{v_1})(\overline{u_2}+\overline{v_2})}(-1)^{\overline{v_2}(\overline{u_1}+\overline{v_1})}(u_2u_1)\otimes(v_1v_2)
\]
\[
=(-1)^{\overline{v_1}(\overline{u_2}+\overline{v_2})}(u_1u_2)\otimes(v_2v_1)-(-1)^{(\overline{u_1}+\overline{v_1})\overline{u_2}}(u_2u_1)\otimes(v_1v_2)
\]
\[
=((-1)^{\overline{v_1}(\overline{u_2}+\overline{v_2})}
-(-1)^{(\overline{u_1}+\overline{v_1})\overline{u_2}}(-1)^{\overline{u_1}\,\overline{u_2}+\overline{v_1}\,\overline{v_2}})(u_1u_2)\otimes(v_2v_1)=0,
\]
i.e. the super-commutative super-identity also holds for $G^2(X_\text{\rm sup})$.
\end{proof}

\begin{theorem}\label{basis}
Let $F({\mathfrak B}_{\text{\rm sup}})$ be the free bicommutative superalgebra generated by the set  $X$ over a field
of characteristics different from two. Then the set $X\cup M$, where $M$ is defined in {\rm (\ref{basis of superalgebra})}, forms a basis of $F({\mathfrak B}_{\text{\rm sup}})$.
The superalgebra $F({\mathfrak B}_{\text{\rm sup}})$ is isomorphic to the superalgebra $G_{\text{\rm sup}}(X)$ defined above Lemma {\rm \ref{the algebra G}}.
\end{theorem}

\begin{proof}
By the universal property of the free bicommutative superalgebra, if $A$ is an arbitrary bicommutative superalgebra,
then every preserving the ${\mathbb Z}_2$-grading mapping $X\to A$ is extended in a unique way to a homomorphism $F({\mathfrak B}_{\text{\rm sup}})\to A$.
By Lemma \ref{the algebra G} this holds also for the superalgebra $G_{\text{\rm sup}}(X)$ generated by the set $X$.
Hence the homomorphism $F({\mathfrak B}_{\text{\rm sup}})\to G_{\text{\rm sup}}(X)$ defined by $x\to x$, $x\in X$, is onto. Clearly, the element
\[
y_1(\cdots(y_k(z_1(\cdots(z_{l-1}((\cdots(((\cdots(z_ly_{k+1})\cdots)y_{k+m})z_{l+1})\cdots)z_{l+n}))\cdots)))\cdots)
\]
from (\ref{basis of superalgebra}) maps to
\begin{equation}\label{elements of G}
y_1\cdots y_kz_1\cdots z_{l-1} z_l\otimes y_{k+1}\cdots y_{k+m}z_{l+1}\cdots z_{l+n}.
\end{equation}
Since the elements in (\ref{elements of G}) are linearly independent in $G^2(X_\text{\rm sup})$,
the same holds for their preimages  (\ref{basis of superalgebra}) in $F^2({\mathfrak B}_{\text{\rm sup}})$.
This implies that the pre-basis $X\cup M$ of $F({\mathfrak B}_{\text{\rm sup}})$ is a basis.
\end{proof}

\begin{remark}\label{sypercommutativity of square of superalgebra}
As a combination of Lemma \ref{the algebra G} (ii) and Theorem \ref{basis} we obtain that the square $A^2$ of any bicommutative superalgebra $A$ is associative super-commutative.
\end{remark}

\section{Hilbert series of free bicommutative superalgebras}

In this section we fix a finite set
\[
X_{p,q}=Y_p\cup Z_q,\quad Y_p=\{y_1,\ldots,y_p\},Z_q=\{z_1,\ldots,z_q\}
\]
and consider the free bicommutative superalgebra $F_{p,q}({\mathfrak B}_{\text{\rm sup}})$ freely generated by $X_{p,q}$.
The algebra $F_{p,q}({\mathfrak B}_{\text{\rm sup}})$ is ${\mathbb N}$-graded:
\[
F_{p,q}({\mathfrak B}_{\text{\rm sup}})=F_{p,q}^{(1)}\oplus F_{p,q}^{(2)}\oplus F_{p,q}^{(3)}\oplus\cdots,
\]
where $F_{p,q}^{(n)}=F_{p,q}^{(n)}({\mathfrak B}_{\text{\rm sup}})$ is the set of homogeneous elements of degree $n$ in $F_{p,q}({\mathfrak B}_{\text{\rm sup}})$.
The {\it Hilbert series} of $F_{p,q}({\mathfrak B}_{\text{\rm sup}})$ is the formal power series
\[
H(F_{p,q}({\mathfrak B}_{\text{\rm sup}}),t)=\sum_{n\geq 1}\dim(F_{p,q}^{(n)})t^n.
\]
The algebra $F_{p,q}({\mathfrak B}_{\text{\rm sup}})$ has two more precise gradings:

${\mathbb N}_0^2$ grading
\[
F_{p,q}({\mathfrak B}_{\text{\rm sup}})=\bigoplus_{k,l}F_{p,q}^{(k,l)},
\]
where ${\mathbb N}_0$ is the set of nonnegative integers and $F_{p,q}^{(k,l)}$ consists of the elements of $F_{p,q}({\mathfrak B}_{\text{\rm sup}})$
which are homogeneous of degree $k$ and $l$ with respect to the even and odd variables, respectively;

${\mathbb N}_0^{p+q}$ grading
\[
F_{p,q}({\mathfrak B}_{\text{\rm
sup}})=\bigoplus_{\text{\bf k}=(k_1,\ldots,k_p), \atop \text{\bf l}=(l_1,\ldots,l_q)}F_{p,q}^{(\text{\bf k},\text{\bf l})},
\]
which counts the degree of each variable in $X_{p,q}$. The corresponding Hilbert series are
\[
H(F_{p,q}({\mathfrak B}_{\text{\rm sup}}),u,v)=\sum_{k+l\geq 1}\dim(F_{p,q}^{(k,l)})u^kv^l,
\]
\[
H(F_{p,q}({\mathfrak B}_{\text{\rm sup}}),U_p,V_q)=\sum_{k_i,l_j}\dim(F_{p,q}^{(\text{\bf k},\text{\bf l})})u_1^{k_1}\cdots u_p^{k_p}v_1^{l_1}\cdots v_q^{l_q}.
\]

\begin{remark}\label{working in G instead of in F}
By Theorem \ref{basis} the free bicommutative superalgebra $F_{p,q}({\mathfrak B}_{\text{\rm sup}})$
and the algebra $G_{\text{\rm sup}}(X_{p,q})$ generated by the set $X_{p,q}=Y_p\cup Z_q$ are isomorphic.
It is easy to see that they are also isomorphic as ${\mathbb N}_0^{p+q}$-graded algebras.
In the sequel we shall often work in $G_{\text{\rm sup}}(X_{p,q})$ instead of in $F_{p,q}({\mathfrak B}_{\text{\rm sup}})$
but shall state the results for $F_{p,q}({\mathfrak B}_{\text{\rm sup}})$.
\end{remark}

\begin{theorem}\label{Hilbert series}
The Hilbert series of the free bicommutative superalgebra $F_{p,q}({\mathfrak B}_{\text{\rm sup}})$ are
\begin{equation}\label{Hilbert multigraded}
H(F_{p,q}({\mathfrak B}_{\text{\rm sup}}),U_p,V_q)=\sum_{i=1}^pu_i+\sum_{j=1}^qv_j+\left(\prod_{i=1}^p\frac{1}{1-u_i}\prod_{j=1}^q(1+v_j)-1\right)^2;
\end{equation}
\begin{equation}\label{Hilbert bigraded}
H(F_{p,q}({\mathfrak B}_{\text{\rm sup}}),u,v)=pu+qv+\left(\frac{(1+v)^q}{(1-u)^p}-1\right)^2;
\end{equation}
\begin{equation}\label{Hilbert graded}
H(F_{p,q}({\mathfrak B}_{\text{\rm sup}}),t)=(p+q)t+\left(\frac{(1+t)^q}{(1-t)^p}-1\right)^2.
\end{equation}
\end{theorem}

\begin{proof}
We shall work in $G_{\text{\rm sup}}(X_{p,q})$. The summand $\displaystyle \sum_{i=1}^pu_i+\sum_{j=1}^qv_j$ in the Hilbert series $H(F_{p,q}({\mathfrak B}_{\text{\rm sup}}),U_p,V_q)$
contains the information for the grading of the vector space $KX_{p,q}$. Hence it is sufficient to calculate the Hilbert series of $G^2_{\text{\rm sup}}(X_{p,q})$.
Since $G^2_{\text{\rm sup}}(X_{p,q})$ and $\omega(F({\mathfrak A}_{\text{\rm sup}}))\otimes_K\omega(F({\mathfrak A}_{\text{\rm sup}}))$ are isomorphic as ${\mathbb N}_0^{p+q}$-graded vector spaces
we have that
\[
H(G^2_{\text{\rm sup}}(X_{p,q}),U_p,V_q)=H^2(\omega(F_{p,q}({\mathfrak A}_{\text{\rm sup}})))=H(F_{p,q}({\mathfrak A}_{\text{\rm sup}}))-1)^2.
\]
The isomorphism $F_{p,q}({\mathfrak A}_{\text{\rm sup}})\cong K[Y_p]\otimes_KF_{0,q}({\mathfrak A}_{\text{\rm sup}}))$
implies that
\[
H(F_{p,q}({\mathfrak A}_{\text{\rm sup}}),U_p,V_q)=H(K[Y_p],U_p)H(F_{0,q}({\mathfrak A}_{\text{\rm sup}}),V_q)
\]
\[
=\prod_{i=1}^p\frac{1}{1-u_i}H(F_{0,q}({\mathfrak A}_{\text{\rm sup}}),V_q).
\]
The superalgebra $F_{0,q}({\mathfrak A}_{\text{\rm sup}})$ has a basis consisting of all monomials
\[
1, z_{j_1}\cdots z_{j_l},\quad 1\leq j_1<\cdots<j_l\leq q
\]
and hence its Hilbert series is
\[
H(F_{0,q}({\mathfrak A}_{\text{\rm sup}}),V_q)=\sum_{1\leq j_1<\cdots<j_l\leq q}v_{j_1}\cdots v_{j_l}=\prod_{j=1}^q(1+v_j).
\]
This completes the proof for the Hilbert series (\ref{Hilbert multigraded}).
The statements for the Hilbert series (\ref{Hilbert bigraded}) and (\ref{Hilbert graded}) follow immediately because
\[
H(F_{p,q}({\mathfrak B}_{\text{\rm sup}}),u,v)=H(F_{p,q}({\mathfrak A}_{\text{\rm sup}})),\underbrace{u,\ldots,u}_{p\text{ times}}\underbrace{v,\ldots,v}_{q\text{ times}}),
\]
\[
H(F_{p,q}({\mathfrak B}_{\text{\rm sup}}),t)=H(F_{p,q}({\mathfrak A}_{\text{\rm sup}})),\underbrace{t,\ldots,t}_{p+q\text{ times}}).
\]
\end{proof}

\begin{corollary}\label{coefficients of Hilbert series}
The dimensions of the homogeneous components of $F_{p,q}({\mathfrak B}_{\text{\rm sup}})$ are
\begin{equation}\label{multigrading}
\dim(F_{p,q}^{(\text{\bf k},\text{\bf l})})=\prod_{i=1}^p(k_i+1)\prod_{j=1}^q(\delta_{2,l_j}+2\delta_{1,l_j}+\delta_{0,l_j})-2\prod_{j=1}^q(\delta_{1,l_j}+\delta_{0,l_j})
\end{equation}
for $\displaystyle \sum_{i=1}^pk_i+\sum_{j=1}^ql_j>1$, where $\delta_{i,j}$ is the Kronecker delta and
\[
\dim(F_{p,q}^{(\text{\bf k},\text{\bf l})})=1\text{ when }\sum_{i=1}^pm_i+\sum_{j=1}^qn_j=1;
\]
\begin{equation}\label{bigrading}
\dim(F_{p,q}^{(k,l)})=\begin{cases}
{\displaystyle
\binom{k+2p-1}{k}\binom{2q}{l}-2\binom{k+p-1}{k}\binom{q}{l}},\text{ if }k+l>1;\\
p,\text{ if }(p,q)=(1,0);\\
q,\text{ if }(p,q)=(0,1).
\end{cases}
\end{equation}
\end{corollary}

\begin{proof}
Expanding the Hilbert series (\ref{Hilbert multigraded}) we obtain
\[
H(F_{p,q}({\mathfrak B}_{\text{\rm sup}}),U_p,V_q)=\sum_{i=1}^pu_i+\sum_{j=1}^qv_j+\left(\prod_{i=1}^p\frac{1}{1-u_i}\prod_{j=1}^q(1+v_j)-1\right)^2
\]
\[
=\sum_{i=1}^pu_i+\sum_{j=1}^qv_j+\prod_{i=1}^p\frac{1}{(1-u_i)^2}\prod_{j=1}^q(1+v_j)^2-2\prod_{i=1}^p\frac{1}{1-u_i}\prod_{j=1}^q(1+v_j)+1
\]
\[
=\sum_{i=1}^pu_i+\sum_{j=1}^qv_j+\sum_{k_i\geq 0}\prod_{i=1}^p(k_i+1)u_1^{k_1}\cdots u_p^{k_p}\prod_{j=1}^q(1+2v_j+v_j^2)
\]
\[
-2\sum_{k_i\geq 0}\prod_{i=1}^pu_1^{k_1}\cdots u_p^{k_p}\prod_{j=1}^q(1+v_j)+1
\]
and this gives the formula (\ref{multigrading}) for $\dim(F_{p,q}^{(\text{\bf k},\text{\bf l})})$.
Similarly, expanding the Hilbert series (\ref{Hilbert bigraded}) we obtain the formula (\ref{bigrading}) for $\dim(F_{p,q}^{(k,l)})$.
\end{proof}

In the theory of algebras with polynomial identities a special role is played by the codimension sequence.
In the case of varieties ${\mathfrak V}_{\text{\rm sup}}$ of superalgebras one considers two sequences:
\[
c_{p,q}({\mathfrak V}_{\text{\rm sup}})=\dim(P_{p,q}({\mathfrak V}_{\text{\rm sup}})),\,p,q=0,1,2,\ldots, p+q>0,
\]
where $P_{p,q}({\mathfrak V}_{\text{\rm sup}})$ is the multilinear component of $F({\mathfrak V}_{\text{\rm sup}})$ in the variables
$y_1,\ldots,y_p$ and $z_1,\ldots,z_q$;
\[
c_n({\mathfrak V}_{\text{\rm sup}})=\dim(P_n({\mathfrak V}_{\text{\rm sup}})),\,n=1,2,\ldots,
\]
where $P_n({\mathfrak V}_{\text{\rm sup}})$ is the multilinear component in $x_1,\ldots,x_n$
and each $x_i$ is equal either to $y_i$ or to $z_i$, $i=1,\ldots,n$. Obviously,
\begin{equation}\label{supercodimension}
c_{p,q}({\mathfrak V}_{\text{\rm sup}})=F_{p,q}^{(1,\ldots,1,1,\ldots,1)}({\mathfrak V}_{\text{\rm sup}}),
\end{equation}
\begin{equation}\label{ordinary codimensions}
c_n({\mathfrak V}_{\text{\rm sup}})=\sum_{p=0}^n\binom{n}{p}c_{p,n-p}({\mathfrak V}_{\text{\rm sup}}).
\end{equation}

\begin{corollary}\label{codimensions}
\begin{equation}\label{supercodim}
c_{p,q}({\mathfrak B}_{\text{\rm sup}})=\begin{cases}
2^{p+q}-2,\text{ if }p+q>1;\\
1,\text{ if }p+q=1,
\end{cases}
\end{equation}
\begin{equation}\label{codim}
c_n({\mathfrak B}_{\text{\rm sup}})=\begin{cases}
2^{2n}-2^{n+1},\text{ if }n>1;\\
2,\text{ if }n=1.
\end{cases}
\end{equation}
\end{corollary}

\begin{proof}
In virtue of (\ref{supercodimension}) the formula (\ref{supercodim}) for $c_{p,q}({\mathfrak B}_{\text{\rm sup}})$
is an immediate consequence of (\ref{multigrading}) when $k_1=\cdots=k_p=l_1=\cdots=l_q=1$.
Applying (\ref{ordinary codimensions}) and (\ref{supercodim}) we obtain
\[
c_n({\mathfrak B}_{\text{\rm sup}})=\sum_{p=0}^n\binom{n}{p}c_{p,n-p}({\mathfrak B}_{\text{\rm sup}})=\sum_{p=0}^n\binom{n}{p}(2^{p+q}-2)=2^n(2^n-2)
\]
for $n>1$ and $c_1({\mathfrak B}_{\text{\rm sup}})=2$ when $n=1$.
\end{proof}

\section{Finitely generated bicommutative superalgebras}

As we mentioned in the introduction, one of the leading ideas of our paper is the following.
Since the square $F^2({\mathfrak B}_{\text{\rm sup}})$ of the free bicommutative superalgebra $F({\mathfrak B}_{\text{\rm sup}})$
is an associative super-commutative superalgebra, we want to transfer to the superalgebra $F({\mathfrak B}_{\text{\rm sup}})$
properties typical for associative super-commutative superalgebras.
In what follows we shall state the classical result and immediately we shall prove its bicommutative superalgebra analogue.
For a background on commutative algebra see e.g. \cite{Atiyah-Macdonald}.

The following easy example in \cite{Drensky-Zhakhaev} shows that there are also some differences.

\begin{example}
The free bicommutative algebra $F_{1,0}=F_{1,0}({\mathfrak B}_{\text{\rm sup}})=F_1({\mathfrak B})$ is not noetherian. Its left ideal generated by the monomials
\[
y_1y_1^n,\quad n=1,2,\ldots,
\]
is not finitely generated.
\end{example}

But we shall show that the situation is completely different if we consider two-sided ideals and the picture is similar with the commutative and super-commutative case.

\subsection{Hilbert Basissatz}
The classical Hilbert Basissatz states:

\begin{theorem}\label{classical Hilbert Basissatz}
Every finitely generated associative commutative algebra is noetherian, i.e. it satisfies the ascending chain condition for its ideals.
\end{theorem}

The bicommutative super-counterpart is the following.

\begin{theorem}\label{Hilbert Basissatz}
Finitely generated bicommutative superalgebras are weakly noetherian, i.e. they satisfy the ascending chain condition for two-sided ideals.
\end{theorem}

\begin{proof}
It is sufficient to work in the square $F_{p,q}^2({\mathfrak B}_{\text{\rm sup}})$ of the finitely generated free bicommutative superalgebra $F_{p,q}({\mathfrak B}_{\text{\rm sup}})$.
Instead, it is more convenient to work in the tensor square $\omega(F_{p,q}({\mathfrak A}_{\text{\rm sup}}))\otimes\omega(F_{p,q}({\mathfrak A}_{\text{\rm sup}}))$.
It is a finitely generated $K[Y_p]$-bimodule generated by the products
\[
y_{i_1}\otimes y_{i_2},\, y_i\otimes z_{j_1}\cdots z_{j_k},\, z_{j_1}\cdots z_{j_k}\otimes y_i,\,z_{j_1}\cdots z_{j_k}\otimes z_{j_{k+1}}\cdots z_{j_{k+l}},
\]
$k,l\geq 1$, $j_1<\cdots <j_k$, $j_{k+1}<\cdots <j_{k+l}$. The multiplication from the right by a monomial $y_{i_1}\cdots y_{i_k}\in K[Y_p]$
is realized by consecutive multiplications from the right by the the generators $y_{i_1},\ldots, y_{i_k}$ of $G_{\text{\rm sup}}(X_{p,q})$,
and similarly we can realize the multiplication from the left. Since the ideals $I\subset G_{\text{\rm sup}}^2(X_{p,q})$ of $G_{\text{\rm sup}}(X_{p,q})$
are $K[Y_p]$-subbimodules of $G^2_{\text{\rm sup}}(X_{p,q})$, they are also finitely generated as $K[Y_p]$-bimodules and hence also as ideals.
\end{proof}

\subsection{Rationality of Hilbert series}
The Hilbert-Serre theorem deals with the Hilbert series of finitely generated graded modules of $K[X_p]$.

\begin{theorem}\label{classical Hilbert-Serre}
{\rm (i)} Let $M$ be a finitely generated graded $K[Y_p]$-module. Then its Hilbert series $H(M,t)$ is a rational function of the form
\[
H(M,t)=f(t)\prod\frac{1}{1-t^{d_i}},
\]
where $f(t)$ is a polynomial with integral coefficients.

{\rm (ii)} If $M$ is ${\mathbb N}_0^p$-graded, then
\[
H(M,t_1,\ldots,t_p)=f(t_1,\ldots,t_p)\prod\frac{1}{1-t_1^{d_{1i}}\cdots t_p^{d_{pi}}}.
\]
\end{theorem}

We shall prove the following.

\begin{theorem}\label{Hilbert-Serre}
{\rm (i)} Every finitely generated graded bicommutative superalgebra $A$ has a rational Hilbert series
\[
H(A,t)=f(t)\prod\frac{1}{1-t^{d_i}},\quad f(t)\in{\mathbb Z}[t].
\]

{\rm (ii)} If $A$ is ${\mathbb N}_0^p\times {\mathbb N}_0^q$-graded, where ${\mathbb N}_0^p$ and ${\mathbb N}_0^q$ count the even and odd gradings,
respectively, then the Hilbert series $H(A,U_p,V_q)$ of $A$ is of the form
\[
H(A,U_p,V_q)=f(U_p,V_q)\prod\frac{1}{1-U_p^{d_i}}.
\]
If $A$ is generated by generators $a_1,\ldots,a_r$ of multidegree $(m_{i1},\ldots,m_{ip},n_{i1},\ldots,n_{iq})$, $i=1,\ldots,r$,
then the polynomial $f(U_p,V_q)$ satisfies
\begin{equation}\label{bound for odd degree}
\deg_{v_j}f(U_p,V_q)\leq 2\sum_{i=1}^rn_{ij}.
\end{equation}
\end{theorem}

\begin{proof}
(ii) We shall show that (ii) follows from Theorem \ref{classical Hilbert-Serre} (ii).
Let the first $k$ generators $a_1,\ldots,a_k$ be of multidegree $(n_{i1},\ldots,n_{iq})=(0,\ldots,0)$, $i=1,\ldots,k$,
and let for the other $l=r-k$ generators $b_{k+1}=a_{k+1},\ldots,b_{k+l}=a_{k+l}$ the multidegrees $(n_{j1},\ldots,n_{jq})$ be different from $(0,\ldots,0)$, $j=k+1,\ldots,k+l$. Then
\[
H(A,U_p,V_q)=\sum_{i=1}^{k+l}u_1^{m_{i1}}\cdots u_p^{m_{ip}}v_1^{n_{i1}}\cdots v_q^{n_{iq}}+H(A^2,U_p,V_q).
\]
The condition $(n_{j1},\ldots,n_{jq})\not=(0,\ldots,0)$ implies that the generator $b_j$ participates not more than two times in each monomial in the generators and this gives that
\begin{equation}\label{Hilbert series of square of A}
H(A^2,U_p,V_q)=\sum_{n_1,\ldots,n_q}f_{(n_1,\ldots,n_q)}(U_p)v_1^{n_1}\cdots v_q^{n_q},
\end{equation}
where the degree $n_j$ of $v_j$ in $H(A^2,U_p,V_q)$ is bounded as in (\ref{bound for odd degree}).
The multigrading of $A$ implies that $A^2$ is a direct sum of ${\mathbb N}_0^p$-graded subspaces $(A^2)^{(n_1,\ldots,n_q)}$
which are multihomogeneous of multidegree $(n_1,\ldots,n_q)$ with respect to the odd grading.
As in the proof of Theorem \ref{Hilbert Basissatz} the multiplication from the right and from the left equips every component $(A^2)^{(n_1,\ldots,n_q)}$
with the structure of a $K[a_1,\ldots,a_k]$-bimodule. By the Hilbert-Serre theorem the functions $f_{(n_1,\ldots,n_q)}(U_p)$ in (\ref{Hilbert series of square of A})
are rational functions with denominators which are products of binomials $1-u_1^{d_{1i}}\cdots u_p^{d_{pi}}$.

(i) The proof of follows from Theorem \ref{classical Hilbert-Serre} (i) with similar arguments as in our proof of (ii).
\end{proof}

\subsection{Gr\"obner-Shirshov bases}
Let the semigroup $[X_p]$ of the monomials in $p$ associative commutative variables be {\it well-ordered} with ordering $\preceq$, which is {\it compatible with the multiplication},
i.e. the ordering $\preceq$ is linear, every subset of $[X_p]$ has a minimal element and if $u\prec v$ in $[X_p]$, then $uw\prec vw$ for any $w\in[X_p]$.
If $0\not=f(X_d)\in K[X_d]$, then we denote by $\text{\rm lead}(f)$ the leading monomial of $f(X_p)$ (with respect to the ordering $\preceq$).
Similarly, if $I$ is an ideal of the polynomial algebra $K[X_p]$, then $\text{\rm lead}(I)$ is the set of the leading monomials of the polynomials in $I$.
Clearly, $\text{\rm lead}(I)$ is an ideal of the multiplicative semigroup $[X_p]$.
The subset $\text{\rm Gr\"obner}(I)$ of the ideal $I$ is a {\it Gr\"obner basis} of $I$ if for every $0\not=f(X_p)\in I$
there exists $f_i(X_p)\in \text{\rm Gr\"obner}(I)$ such that $\text{\rm lead}(f_i)$ divides $\text{\rm lead}(f)$.
It follows immediately from the Hilbert Basissatz \ref{classical Hilbert Basissatz}
applied to the monomial ideal of $K[X_p]$ generated by $\text{\rm lead}(I)$ (or from the Dickson lemma \cite{Dickson})
that every ideal $I$ of $K[X_p]$ has a finite Gr\"obner basis.

\begin{theorem}\label{ordinary Groebner basis}
Any ideal $I$ of $K[X_p]$ is generated by its Gr\"obner basis $\text{\rm Gr\"obner}(I)$ and
the factor algebra $K[X_p]/I$ has a vector space basis consisting of all monomials which are not divisible by any leading monomial of the polynomials in $\text{\rm Gr\"obner}(I)$.
\end{theorem}

Gr\"obner bases were introduced by Buchberger in 1965 in his Ph.D. thesis \cite{Buchberger}.
Ideas, similar to his ideas appeared also in the work of other mathematicians,
some of them even before him, e.g. by G\"unter \cite{Gunter-1}, \cite{Gunter-2} and Hironaka \cite{Hironaka}.
But the significance of the contributions of Buchberger are undeniable.
In particular, his algorithm is in the ground of many of the computations with polynomials in commutative algebra and algebraic geometry.

Investigations in the spirit of those of Buchberger appeared independently also in noncommutative algebra.
In 1962 Shirshov developed in \cite{Shirshov} his algorithmic approach to Lie algebras.
His approach works in much larger cases and nowadays one speaks about {\it Gr\"obner-Shirshov bases}.
The case for associative algebras was developed in 1978 by Bergman in \cite{Bergman}.
See the survey by Bokut and Kolesnikov \cite{Bokut-Kolesnikov} for the history and more new developments in the theory of Gr\"obner-Shirshov bases.
See  also Bai, Chen and Zhang \cite{Bai-Chen-Zhang2022} for the Gr\"obner-Shirshov bases of the ideals of free
bicommutative algebras.

Now we shall introduce the concept of Gr\"obner-Shirshov bases for two-sided ideals of the free bicommutative superalgebra $F_{p,q}({\mathfrak B}_{\text{\rm sup}})$.
It is more convenient to work in the superalgebra $G_{\text{\rm sup}}(X_{p,q})$ instead of in $F_{p,q}({\mathfrak B}_{\text{\rm sup}})$.
Let $[X_{p,q}]$ be the commutative semigroup with zero of the monomials in the set $X_{p,q}=Y_p\cup Z_q$ of $p+q$ commuting variables
with defining relations $z_j^2=0$ for the $q$ ``odd'' variables $z_1,\ldots,z_q$.
We consider the multiplicative semigroup $X_{p,q}\cup[X_{p,q}]\otimes[X_{p,q}]$
with multiplication which up to a sign is the same as in the superalgebra $G_{\text{\rm sup}}(X_{p,q})$, i.e.
\[
x_i\circ x_j=x_i\otimes x_j,\, x_i\circ (u\otimes v)=x_iu\otimes v,\,(u\otimes v)\circ x_i=u\otimes vx_i,
\]
\[
(u_1\otimes v_1)\circ(u_2\otimes v_2)=(u_1u_2)\otimes(v_1v_2).
\]
Let $X_{p,q}\cup[X_{p,q}]\otimes[X_{p,q}]$ be well-ordered with ordering $\preceq$, which is compatible with the multiplication.
In our case this will mean that if $u\prec v$ in $X_{p,q}\cup [X_{p,q}]\otimes[X_{p,q}]$ and $x_iu,x_iv\not=0$ in $X_{p,q}\cup [X_{p,q}]\otimes[X_{p,q}]$ for $x_i\in X_{p,q}$,
then $x_iu\prec x_iv$. Similarly, if $ux_i,vx_i\not=0$, then $ux_i\prec vx_i$. If $0\not=f(X_{p,q})\in G_{\text{\rm sup}}(X_{p,q})$,
then we denote by $\text{\rm lead}(f)$ the leading monomial of $f(X_{p,q})$ (with respect to the ordering $\preceq$).
Similarly, if $I$ is a two-sided ideal of the superalgebra $G_{\text{\rm sup}}(X_{p,q})$, then $\text{\rm lead}(I)$ is the set of the leading monomials of the polynomials in $I$.
Clearly, $\text{\rm lead}(I)$ is an ideal of the semigroup $X_{p,q}\cup [X_{p,q}]\otimes[X_{p,q}]$.
The important property which we shall use is the following and it is easy to be checked. \begin{lemma}\label{propery of leading monomial} If
\[
0\not=f_1(X_{p,q}),f_2(X_{p,q})\in G_{\text{\rm sup}}(X_{p,q}),\,\text{\rm lead}(f_1(X_{p,q}))\prec\text{\rm lead}(f_2(X_{p,q}))
\]
and $x_i\text{\rm lead}(f_2(X_{p,q}))\not=0$, then
\[
\text{\rm lead}(x_if_1(X_{p,q}))\prec\text{\rm lead}(x_if_2(X_{p,q}))\text{ or }x_if_1(X_{p,q})=0.
\]
Similarly, $\text{\rm lead}(f_2(X_{p,q}))x_i\not=0$ implies that
\[
\text{\rm lead}(f_1(X_{p,q}))x_i\prec\text{\rm lead}(f_2(X_{p,q})x_i)\text{ or }f_1(X_{p,q})x_i=0.
\]
\end{lemma}

\begin{definition}\label{definition of Groebner basis}
The subset $\text{\rm Gr\"obner}(I)$ of the two-sided ideal $I$ of $F_{p,q}({\mathfrak B}_{\text{\rm sup}})$ is a {\it Gr\"obner-Shirshov basis} of $I$
if for every $0\not=f(X_{p,q})\in I$ there exists  $f_i(X_{p,q})\in \text{\rm Gr\"obner}(I)$ such that $\text{\rm lead}(f_i)$ divides $\text{\rm lead}(f)$.
The latter means that $\text{\rm lead}(f)$ can be obtained from $\text{\rm lead}(f_i)$ by consecutive multiplications from both sides by free generators from $X_{p,q}$.
\end{definition}

The following theorem is an analogue of Theorem \ref{ordinary Groebner basis} for finitely generated free bicommutative superalgebras.

\begin{theorem} {\rm (i)} Any two-sided ideal $I$ of $F_{p,q}({\mathfrak B}_{\text{\rm sup}})$ has a finite Gr\"obner-Shirshov basis $\text{\rm Gr\"obner}(I)$;

{\rm (ii)} The ideal $I$ is generated by $\text{\rm Gr\"obner}(I)$;

{\rm (iii)} The factor superalgebra $F_{p,q}({\mathfrak B}_{\text{\rm sup}})/I$ has a vector space basis consisting of all monomials
which are not divisible by any leading monomial of the polynomials in $\text{\rm Gr\"obner}(I)$.
\end{theorem}

\begin{proof} (i) As in the case of commutative algebras in Theorem \ref{ordinary Groebner basis}
the existence of a finite Gr\"obner-Shirshov basis follows from the Hilbert Basissatz \ref{Hilbert Basissatz}.

(ii) The proof follows from Lemma \ref{propery of leading monomial} applying the same arguments as in the case of ordinary commutative algebras.

(iii) The proof is similar to the proof in the case of usual polynomial algebras.
\end{proof}

\subsection{Gelfand-Kirillov dimension}
Let $A$ be an arbitrary finitely generated algebra and let $V$ be the subspace of $A$ spanned by an arbitrary finite system of generators.
The {\it growth function} of $A$ with respect to $V$ is
\[
g_V(n)=\dim(V^0+V^1+\cdots+V^n),\,n=1,2,\ldots,
\]
where $V^0=K$ for unitary algebras and $V^0=0$ for algebras without 1
and $V^n$ is spanned by all products of $n$ elements of $V$ with arbitrary decomposition of the parentheses.
The {\it Gelfand-Kirillov dimension} of $A$ is defined as
\[
\text{GKdim}(A)=\limsup_{n\to\infty}\log_n(g_V(n))
\]
and does not depend on the choice of the finite generating system of $A$, see \cite{Krause-Lenagan}.
For graded algebras with rational Hilbert series the information for the Hilbert series is sufficient to determine the Gelfand-Kirillov dimension,
see \cite[Remark, page 91 of the Russian original and page 87 of the translation]{Ufnarovskij}.

\begin{theorem}\label{GKdim and rational Hilbert series}
If the Hilbert series of an algebra is a rational function with polynomial growth of the coefficients and has a pole of the $m$-th order for $t=1$,
then the Gelfand-Kirillov dimension of the algebra is equal to $m$.
\end{theorem}

The proof of the following theorem follows immediately from Theorems \ref{Hilbert-Serre} and \ref{GKdim and rational Hilbert series}.

\begin{theorem}
The Gelfand-Kirillov dimension of every finitely generated bicommutative superalgebra is a nonnegative integer.
\end{theorem}

\begin{example}
As a direct consequence of (\ref{Hilbert graded})
in Theorem \ref{Hilbert series} and Theorem \ref{GKdim and rational Hilbert series} we obtain that
\[
\text{\rm GKdim}(F_{p,q}({\mathfrak B}_{\text{\rm sup}}))=2p.
\]
\end{example}

\begin{remark}
In \cite{Bai-Chen-Zhang2022} Bai, Chen and Zhang give an algorithm which allows to compute the Gelfand-Kirillov dimension of the factor-algebra $K[Y_p]/I$
if we know the leading monomials of the Gr\"obner basis of $I$.
If we identify $F_{p,q}({\mathfrak B}_{\text{\rm sup}})$ with the superalgebra $G_{\text{\rm sup}}(X_{p,q})$,
then a similar algorithm holds for the ideals of $F_{p,q}({\mathfrak B}_{\text{\rm sup}})$
if we take into account only the even part $y_1\cdots y_k\otimes y_{k+1}\cdots y_{k+m}$ of the monomials
\[
y_1\cdots y_kz_1\cdots z_{l-1} z_l\otimes y_{k+1}\cdots y_{k+m}z_{l+1}\cdots z_{l+n}
\]
in the Gr\"obner-Shirshov basis of the ideal.
\end{remark}

\begin{remark}
Since bicommutative algebras may be considered as bicommutative superalgebras with trivial odd component,
the results of this section hold also for finitely generated bicommutative algebras.
In particular, we obtain new proofs of some of the results in \cite{Bai-Chen-Zhang2022}.
Most of the proofs in \cite{Bai-Chen-Zhang2022} are based on methods typical for nonassociative ring theory
but in our approach we use classical results of commutative algebra.
\end{remark}

\section{The Specht property}

In this section we shall prove that every variety of bicommutative superalgebras has a finite basis of its polynomial super-identities. We shall follow the ideas in \cite{Drensky-Zhakhaev}
and shall work in the free bicommutative superalgebra $F({\mathfrak B}_{\text{\rm sup}})$ generated by the set of free generators $X=Y\cup Z$, where $Y=\{y_1,y_2,\ldots\}$, $Z=\{z_1,z_2,\ldots\}$.
But we shall go a little bit
further establishing as in \cite{Drensky-La Scala} the finite generation of any ideal of $F({\mathfrak B}_{\text{\rm sup}})$
in the class of ideals which are closed under the action of the semigroups of order preserving
mappings $Y\to Y$ and $Z\to Z$. As in \cite{Drensky-Zhakhaev} we shall apply the Higman-Cohen method based on the technique of partially ordered sets.

\begin{definition}\label{partially ordered set}
The partially ordered set
$(T, \preceq)$ is {\it partially well-ordered} if for every subset $S$ of $T$ there is a finite subset $S_0$ of $S$ with the property that for each $s\in S$
there is an element $s_0\in S_0$ such that $s_0\preceq s$.
\end{definition}

Let the commutative multiplicative semigroup with identity $[Y]$  be partially ordered by
\[
u_1(Y)= y_1^{a_1}\cdots y_k^{a_k}\preceq y_1^{b_1}\cdots y_m^{b_l}=u_2(Y),\,a_i,b_i\geq 0,
\]
if there exists a sequence $1\leq i_1<i_2<\cdots<i_k\leq l$, such that $u_1(y_{i_1},\ldots,y_{i_k})$ divides $u_2(y_1,\ldots,y_l)$.
In other words, if $\Phi_Y:{\mathbb N}\to{\mathbb N}$ is the semigroup of order preserving maps of $\mathbb N$
and $\Phi_Y$ acts on $[Y]$ as a semigroup of endomorphisms by
\[
\varphi(y_i)=y_{\varphi(i)},\, i=1,2,\ldots,\, \varphi\in\Phi_Y,
\]
then $u_1(Y)\preceq u_2(Y)$ if $\varphi(u_1(Y))$ divides $u_2(Y)$ for some $\varphi\in\Phi_Y$.

We extend the partial ordering $\preceq$ on $[Y]\otimes[Y]$ by
\[
u_1(Y)\otimes v_1(Y)\preceq u_2(Y)\otimes v_2(Y),
\]
if $\varphi(u_1(Y))\otimes \varphi(v_1(Y))$ divides $u_2(Y)\otimes v_2(Y)$ for some $\varphi\in\Phi_Y$.

If $[Z]$ is the commutative multiplicative semigroup with zero and identity and defining relations $z_i^2=0$, $i=1,2,\ldots$,
then we introduce a partial ordering $\preceq$ on $[Z]$ and $[Z]\otimes [Z]$ similar to the partial ordering on $[Y]$ and $[Y]\otimes[Y]$, respectively,
and $\Phi_Z$ is the corresponding semigroup of endomorphisms of $[Z]$.
We also order partially the semigroup $([Y]\times [Z])\otimes([Y]\times [Z])$:
\[
(u_1(Y),v_1(Z))\otimes(u_2(Y),v_2(Z))\preceq(u_3(Y),v_3(Z))\otimes(u_4(Y),v_4(Z))
\]
if $u_1(Y)\otimes u_2(Y)\preceq u_3(Y)\otimes u_4(Y)$ in $([Y]\otimes[Y],\preceq)$ and $v_1(Z)\otimes v_2(Z)\preceq v_3(Z)\otimes v_4(Z)$
in $([Z]\otimes[Z],\preceq)$. Again, the semigroup $\Phi=\Phi_Y\times\Phi_Z$ acts on $G_{\text{\rm sup}}(X)$
and the partial ordering $\preceq$ can be defined in terms of the action of $\Phi$ on $([Y]\times [Z])\otimes([Y]\times[Z])$.

\begin{proposition}\label{lemma Bryant-Vaughan-Lee}
The partially ordered semigroups $([Y]\otimes[Y],\preceq)$, $([Z]\otimes[Z],\preceq)$ and $(([Y]\times [Z])\otimes([Y]\times [Z]),\preceq)$ are partially well-ordered.
\end{proposition}

\begin{proof}
The partial well-ordering of $([Y]\otimes[Y],\preceq)$ and $([Z]\otimes[Z],\preceq)$ follows immediately from \cite[Lemma 1]{Bryant-Vaughan-Lee},
see also \cite[Proposition 4.5]{Drensky-Zhakhaev}. For the partial well-ordering of $(([Y]\times [Z])\otimes([Y]\times [Z]),\preceq)$
it is sufficient to show that the semigroup satisfies the descending chain condition and does not contain an infinite set of incomparable elements. If
\[
(u_{11}(Y),v_{11}(Z))\otimes(u_{12}(Y),v_{12}(Z))\succeq (u_{21}(Y),v_{21}(Z))\otimes(u_{22}(Y),v_{22}(Z))\succeq\cdots,
\]
then $u_{11}(Y)\otimes u_{12}(Y)\succeq u_{21}(Y)\otimes u_{22}(Y)\succeq\cdots$. It follows from the partial well-ordering of $([Y]\otimes[Y],\preceq)$
that $u_{n1}(Y)\otimes u_{n2}(Y)=u_{n+1,1}(Y)\otimes u_{n+1,2}(Y)=\cdots$ for a sufficiently large $n$. Since
$v_{n1}(Z)\otimes v_{n2}(Z)\succeq v_{n+1,1}(Z)\otimes v_{n+1,2}(Z)\succeq\cdots$ and $([Z]\otimes[Z],\preceq)$ is partially well-ordered, we obtain that
$v_{m1}(Z)\otimes v_{m2}(Z)= v_{m+1,1}(Z)\otimes v_{m+1,2}(Z)=\cdots$ for some $m\geq n$.
Hence $(([Y]\times [Z])\otimes([Y]\times [Z]),\preceq)$ satisfies the descending chain condition.
Now, let the infinite set
\[
S=\{(u_{n1}(Y),v_{n1}(Z))\otimes(u_{n2}(Y),v_{n2}(Z))\mid n=1,2,\ldots\}
\]
consist of incomparable elements. Since $([Y]\otimes[Y],\preceq)$ is partially well-ordered, we obtain that the set $S$ contains a subset
\[
S_0=\{(u_{n_i1}(Y),v_{n_i1}(Z))\otimes(u_{n_i2}(Y),v_{n_i2}(Z))\mid n_1<n_2<\cdots\}
\]
such that $u_{n_11}(Y)\otimes u_{n_12}(Y)\preceq u_{n_21}(Y)\otimes u_{n_22}(Y)\preceq\cdots$.
Now the partial well-ordering of $([Z]\otimes[Z],\preceq)$ implies
that the set $\{v_{n_11}(Z)\otimes v_{n_12}(Z),v_{n_21}(Z)\otimes v_{n_22}(Z),\ldots\}$
contains an ordered pair $v_{n_p1}(Z)\otimes v_{n_p2}(Z)\preceq v_{n_q1}(Z)\otimes v_{n_q2}(Z)$, $p<q$. Hence
\[
(u_{n_p1}(Y),v_{n_p1}(Z))\otimes (u_{n_p2}(Y),v_{n_p2}(Z))\preceq (u_{n_q1}(Y),v_{n_q1}(Z))\otimes (u_{n_q2}(Y),v_{n_q2}(Z))
\]
and the set $S$ contains comparable elements.
\end{proof}

In the sequel we shall consider the set $([Y]\times [Z])\otimes([Y]\times [Z])$ equipped with the above
partial ordering $\preceq$. Now we shall define one more linear ordering $\leq$ on $[Y]\otimes[Y]$ and $[Z]\otimes[Z]$
and then we shall extend it on $([Y]\times [Z])\otimes([Y]\times [Z])$. If
\[
Y^a\otimes Y^b=y_1^{a_1}\cdots y_m^{a_m}\otimes y_1^{b_1}\cdots y_m^{b_m},\,
Y^c\otimes Y^d=y_1^{c_1}\cdots y_m^{c_m}\otimes y_1^{d_1}\cdots y_m^{d_m},
\]
then $Y^a\otimes Y^b< Y^c\otimes Y^d$ if:

$a_i<c_i$ for some $i$ and $a_j=c_j$
for $j=i+1,\ldots,m$,

or $Y^a=Y^c$, $b_i<d_i$ for some $i$ and $b_j=d_j$ for $j=i+1,\ldots,m$.

Similarly we order $[Z]\otimes[Z]$ and extend the ordering $\leq$ to $([Y]\times [Z])\otimes([Y]\times [Z])$ by
\[
(u_1(Y),v_1(Z))\otimes(u_2(Y),v_2(Z))< (u_3(Y),v_3(Z))\otimes(u_4(Y),v_4(Z))
\]
\[
\text{if }u_1(Y)\otimes u_2(Y)<u_3(Y)\otimes u_4(Y)
\]
\[
\text{or }u_1(Y)\otimes u_2(Y)=u_3(Y)\otimes u_4(Y)\text{ and }v_1(Z)\otimes v_2(Z)<v_3(Z)\otimes v_4(Z).
\]
Obviously the set $([Y]\times [Z])\otimes([Y]\times [Z])$ is well-ordered. If
\[
f(Y,Z)=\sum_{i=1}^k\vartheta_{(a^{(i)},b^{(i)},c^{(i)},d^{(i)})}Y^{a^{(i)}}Z^{b^{(i)}}\otimes Y^{c^{(i)}}Z^{d^{(i)}},
\quad 0\not=\vartheta_{(a^{(i)},b^{(i)},c^{(i)},d^{(i)})}\in K,
\]
and $Y^{a^{(1)}}Z^{b^{(1)}}\otimes Y^{c^{(1)}}Z^{d^{(1)}}>\cdots >Y^{a^{(k)}}Z^{b^{(k)}}\otimes Y^{c^{(k)}}Z^{d^{(k)}}$, then we call the monomial
\[
\text{wt}(f(Y,Z))=Y^{a^{(1)}}Z^{b^{(1)}}\otimes Y^{c^{(1)}}Z^{d^{(1)}}
\]
the {\it weight} of $f(Y,Z)$. It is easy to see, that if $f(Y,Z)$ belongs to the square $G^2_{\text{\rm sup}}(X)$
of the superalgebra $G_{\text{\rm sup}}(X)$ and $x\circ \text{\rm wt}(f(Y,Z))\not=0$ for $x\in X=Y\cup Z$, then
\begin{equation}\label{multiplication of weight-1}
\text{wt}(x\circ f(Y,Z))=x\circ\text{wt}(f(Y,Z)).
\end{equation}
Similarly, if $\text{\rm wt}(f(Y,Z))\circ x\not=0$, then
\begin{equation}\label{multiplication of weight-2}
\text{wt}(f(Y,Z)\circ x)=\text{wt}(f(Y,Z))\circ x.
\end{equation}
For all $\varphi\in\Phi=\Phi_Y\times \Phi_Z$ we have that
\begin{equation}\label{action on weight by varphi}
\text{wt}(\varphi(f(Y,Z)))=\varphi(\text{wt}(f(Y,Z))).
\end{equation}

\begin{lemma}\label{key lemma}
Let $f(Y,Z)$ and $g(Y,Z)$ be two nonzero polynomials in the square
$G^2_{\text{\rm sup}}(X)$ of the bicommutative superalgebra $G_{\text{\rm sup}}(X)$ and let $I$ be the ideal of $G_{\text{\rm sup}}(X)$ generated by all $\varphi(f(Y,Z))$,
$\varphi\in\Phi=\Phi_Y\times \Phi_Z$. If $\text{\rm wt}(f(Y,Z))\preceq \text{\rm wt}(g(Y,Z))$,
then there is an element $h(Y,Z)$ in the ideal $I$ such that $\text{\rm wt}(h)=\text{\rm wt}(g)$.
\end{lemma}

\begin{proof}
Let
\[
\text{\rm wt}(f)=Y^{a'}Z^{b'}\otimes Y^{c'}Z^{d'},\,\text{\rm wt}(g)=Y^{a''}Z^{b''}\otimes Y^{c''}Z^{d''}
\]
and let $\varphi\in\Phi$ be such that the monomial
$\varphi(\text{\rm wt}(f))=\varphi(Y^{a'}Z^{b'}\otimes Y^{c'}Z^{d'})$
divides the monomial $Y^{a''}Z^{b''}\otimes Y^{c''}Z^{d''}$.
Hence there exists a monomial 
\[
Y^mZ^n\otimes Y^pZ^q\in ([Y]\times[Z])\otimes([Y]\times[Z])
\]
such that
\[
(Y^mZ^n\otimes Y^pZ^q)\circ \varphi(Y^{a'}Z^{b'}\otimes Y^{c'}Z^{d'})=\pm\text{\rm wt}(g).
\]
Since the ideal $I$ is closed under the endomorphism $\varphi$ of $G_{\text{\rm sup}}(X)$, we obtain that $\varphi(f)$ belongs to $I$.
By (\ref{action on weight by varphi}) $\text{wt}(\varphi(f))=\varphi(\text{wt}(f))$.
The weight of $\text{wt}(g)$ can be obtained from $\varphi(\text{wt}(f))$ by consequent right- and left-multiplications by elements of $X$.
The same multiplications produce a polynomial $h\in I$. By (\ref{multiplication of weight-1}) and (\ref{multiplication of weight-2})
\[
(Y^mZ^n\otimes Y^pZ^q)\circ\varphi(\text{\rm wt}(f))=\text{\rm wt}(h)=\text{\rm wt}(g).
\]
\end{proof}

\begin{proposition}\label{f.g.}
Let $I$ be an ideal of $G_{\text{\rm sup}}(X)$ which is a $\Phi$-ideal,
i.e. it is invariant under the action of the semigroup $\Phi_Y\times \Phi_Z$. Then $I$ is finitely generated as a $\Phi$-ideal.
\end{proposition}

\begin{proof}
Let $\overline{I}=I/(I\cap G^2_{\text{\rm sup}}(X))$ be the vector space of the elements of $I$ modulo $G^2_{\text{\rm sup}}(X)$ and let $\overline{I}\not=0$.
The nonzero elements of $\overline{I}$ are of the form $u_{\alpha}(Y)+v_{\beta}(Z)$, 
\[
u_{\alpha}(Y)=\sum_{i=1}^{n_{\alpha}}\alpha_iy_i,\,v_{\beta}(Z)=\sum_{j=1}^{n_{\beta}}\beta_jz_j,\,\alpha_i,\beta_j\in K,
\]
where $u_{\alpha}(Y)\not=0$ (and $\alpha_{n_{\alpha}}\not=0$) or $v_{\beta}(Z)\not=0$ (and $\beta_{n_{\beta}}\not=0$).   

Let $\text{\rm wt}(u_{\alpha}(Y))=y_{n_{\alpha}}$, let $y_m$ be the minimal $\text{\rm wt}(u_{\alpha}(Y))$ for
$u_{\alpha}(Y)+v_{\beta}(Z)\in\overline{I}$, $u_{\alpha}(Y)\not=0$, and let
\[
f_{\beta'}=u_{\alpha'}(Y)+v_{\beta'}(Z)=\alpha'_my_m+\cdots+v_{\beta'}(Z)\in\overline{I},\,\alpha'_m\not=0.
\]
Subtracting several times from any $u_{\alpha}(Y)+v_{\beta}(Z)$ suitable $\gamma\varphi_Y(f_{\beta'})$, $\varphi_Y\in\Phi_Y$, $\gamma\in K$,
we obtain that the generators of $\overline{I}$ can be reduced to elements
\[
\sum_{i=1}^m\alpha_iy_i+v_{\beta}(Z).
\]
For each $i=1,\ldots,m$, we choose $f_i=u_i(Y)+v_i(Z)\in \overline{I}$
with $\text{\rm wt}(u_i(Y))=y_i$ and a minimal $\text{\rm wt}(v_i(Z))=n_i$.
Let $n=\max\{n_1,\ldots,n_m\}$. Then by similar arguments we conclude
that under the action of $\Phi$ the ideal $\overline{I}$ is generated by elements of the form
\[
u_{\alpha}(Y)+v_{\beta}(Z)=\sum_{i=1}^m\alpha_iy_i+\sum_{j=1}^n\beta_jz_j.
\]
Since these $u_{\alpha}(Y)+v_{\beta}(Z)$ span a finite dimensional vector space, the $\Phi$-ideal $\overline{I}$ is finitely generated.
Hence we have to establish the finite generation of the $\Phi$-ideals $I$ which are in $G^2_{\text{\rm sup}}(X)$.
By Proposition \ref{lemma Bryant-Vaughan-Lee} the set of weights
\[
\{\text{\rm wt}(f(Y,Z))\mid f(Y,Z)\in I\}
\]
is partially well-ordered. Let
\[
\{\text{\rm wt}(f_1(Y,Z)),\ldots,\text{\rm wt}(f_n(Y,Z))\},\, f_1,\ldots,f_n\in I,
\]
be the finite subset of the minimal weights of the elements in $I$.
We shall complete the proof if we establish that $f_1,\ldots,f_n$ generate the $\Phi$-ideal $I$.
Let $J$ be the $\Phi$-ideal generated by the polynomials $f_1,\ldots,f_n$ and let us assume that $J\subsetneqq I$.
Let $g\in I\setminus J$ be such that its weight $\text{\rm wt}(g)$ is minimal in the set $\{\text{\rm wt}(f)\mid f\in I\setminus J\}$.
Clearly, $\text{\rm wt}(f_i)\preceq\text{\rm wt}(g)$ for some $f_i$. By Lemma \ref{key lemma} there is an element $h$ in $J$
such that $\text{\rm wt}(h)=\text{\rm wt}(g)$.
If $h\not=g$, then $\text{\rm wt}(g-h)\prec\text{\rm wt}(g)$ and $g-h\in J$ which is a contradiction. Hence $J=I$.
\end{proof}

\begin{remark}
If we consider the weights of the elements of a $\Phi$-ideal $I$ of the free bicommutative superalgebra $F({\mathfrak B}_{\text{\rm sup}})$
we can develop the theory of Gr\"obner-Shirshov bases in the class of $\Phi$-ideals in the spirit of \cite{Drensky-La Scala}.
Proposition \ref{f.g.} shows that $\Phi$-ideals have always finite $\Phi$-Gr\"obner-Shirshov bases.
\end{remark}

The following theorem is the main result of this section.

\begin{theorem}\label{Specht problem for our superalgebras}
The variety ${\mathfrak B}_{\text{\rm sup}}$ of all bicommutative superalgebras over any field of characteristic different from $2$ satisfies the Specht property,
i.e. every variety of bicommutative
superalgebras can be defined by a finite number of super-identities.
\end{theorem}

\begin{proof}
Since the variety of all bicommutative superalgebras is defined by two super-identities
it is sufficient to show that the T-ideal of super-identities $T({\mathfrak V}_{\text{\rm sup}})\subset F({\mathfrak B}_{\text{\rm sup}})$
of any subvariety ${\mathfrak V}_{\text{\rm sup}}$ of ${\mathfrak B}_{\text{\rm sup}}$ is finitely generated.
Identifying the superalgebras $F({\mathfrak B}_{\text{\rm sup}})$ and $G_{\text{\rm sup}}(X)$
it is sufficient to show that the image of $T({\mathfrak V}_{\text{\rm sup}})$ in $G_{\text{\rm sup}}(X)$
is finitely generated in the class of all ideals which are closed under the action of all endomorphisms of $G_{\text{\rm sup}}(X)$.
Hence the proof follows immediately from Proposition \ref{f.g.}.
\end{proof}

Let $A$ be an associative algebra and let $V$ be a vector space which generates $A$ as an algebra.
The element $f(x_1,\ldots,x_n)$ of the free associative algebra $K\langle X\rangle$ is a {\it weak polynomial identity for the pair} $(A,V)$ if
\[
f(v_1,\ldots,v_n)=0\text{ for all }v_1,\ldots,v_n\in V.
\]
The set $T(A,V)$ of all weak polynomial identities of the pair $(A,V)$ is called a {\it weak T-ideal}
and is an ideal of $K\langle X\rangle$ closed under all linear endomorphisms of $K\langle X\rangle$.
(This means that if $f(x_1,\ldots,x_n)\in T(A,V)$ and $u_1,\ldots,u_n$ are linear combinations of elements in $X$, then $f(u_1,\ldots,u_n)\in T(A,V)$.)
Weak polynomial identities were introduced by Razmyslov \cite{Ra1, Ra2} in his approach to solve two important problems:
to find the polynomial identities of the algebra $M_2(K)$ of $2\times 2$ matrices over a field $K$ of characteristic 0
and to prove that for any $d>2$ there exist central polynomials for $M_d(K)$.
In the original approach of Razmyslov the subspace $V$ has an additional structure of a Lie algebra,
i.e. it is closed with respect to the commutator $[v_1,v_2]=v_1v_2-v_2v_1$, $v_1,v_2\in V$, and the weak T-ideal is closed under Lie substitutions.
Similarly, if $V$ is a Jordan algebra under the operation $v_1\circ v_2=v_1v_2+v_2v_1$, $v_1,v_2\in V$,
then the weak T-ideal is closed under Jordan substitutions.

By analogy we introduce the notion of weak polynomial super-identities for bicommutative superalgebras.
If $A$ is a bicommutative superalgebra and $V=V_0\oplus V_1$ is a ${\mathbb Z}_2$-graded vector subspace which generates $A$ as
a superalgebra, then, as in the associative case,
the element $f(Y,Z)\in F({\mathfrak B}_{\text{\rm sup}})$ is called a {\it weak polynomial super-identity for the pair $(A,V)$}
if $\varphi(f(Y,Z))=0$ for all homomorphisms $\varphi:
F({\mathfrak B}_{\text{\rm sup}})\to A$ such that $\varphi(Y)\subset V_0$ and $\varphi(Z)\subset V_1$.
As in the proof of Theorem \ref{Specht problem for our superalgebras} we obtain immediately:

\begin{corollary}\label{weak PI}
For any bicommutative superalgebra $A$ and a ${\mathbb Z}_2$-graded vector subspace $V$ which generates $A$,
the weak polynomial super-identities of the pair $(A,V)$ follow from a finite number.
\end{corollary}

\section{Cocharacters}
In this section we shall assume that $K$ is a field of characteristic 0
and shall describe the cocharacter sequence of the variety of bicommutative superalgebras ${\mathfrak B}_{\text{\rm sup}}$.
We shall use the method of Giambruno \cite{Giambruno} for algebras with involution
which allows to give a relation between the cocharacter sequence and the Hilbert series of $F_{p,q}({\mathfrak B}_{\text{\rm sup}})$.
The method is based on similar considerations of Berele \cite{Berele} and Drensky \cite{Drensky1} for ordinary polynomial identities.
Recall that the irreducible characters of the symmetric group $S_n$ are indexed by partitions of $n$.
If $\lambda=(\lambda_1,\ldots,\lambda_n)\vdash n$ is a partition of $n$, then the corresponding irreducible character is $\chi_{\lambda}$.
The irreducible polynomial representations of the general linear group $GL_d(K)$ are also described
by partitions $\lambda=(\lambda_1,\ldots,\lambda_d)$ in not more than $d$ parts.
The role of the corresponding character is played by the Schur function $s_{\lambda}(T_d)$.
We shall also use the standard notation $\mu'$ for the conjugate partition of $\mu$.
In particular, if $\lambda,\mu,\lambda',\mu'$ are partitions in not more than $d$ parts and
\begin{equation}\label{conjugate partitions}
s_{\lambda}(T_d)s_{\mu}(T_d)=\sum_{\nu}m(\nu)s_{\nu}(T_d),\text{ then }s_{\lambda'}(T_d)s_{\mu'}(T_d)=\sum_{\nu}m(\nu)s_{\nu'}(T_d).
\end{equation}
For a background on the theory of symmetric functions we refer to Macdonald \cite{Macdonald} or Stanley \cite{Stanley}.
We shall use the Young rule for the multiplication of Schur functions $s_{\lambda}(T_d)s_{(n)}(T_d)$.

\begin{theorem}\label{Young rule}
If $\lambda=(\lambda_1,\ldots,\lambda_d)\vdash m$, then
\begin{equation}\label{Young rule 1}
s_{\lambda}(T_d)s_{(n)}(T_d)=\sum_{\mu}s_{\mu}(T_d),
\end{equation}
where $\mu=(\mu_1,\ldots,\mu_d)\vdash m+n$ and
$\mu_1\geq\lambda_1\geq\mu_2\geq\lambda_2\geq\cdots\geq\mu_d\geq\lambda_d$.
\end{theorem}

A combination of Theorem \ref{Young rule} and (\ref{conjugate partitions}) gives a version of the Young rule for $s_{\lambda}(T_d)s_{(1^n)}(T_d)$:
\begin{equation}\label{Young rule 2}
s_{\lambda}(T_d)s_{(1^n)}(T_d)=\sum_{\mu}s_{\mu}(T_d),
\end{equation}
where $n\leq d$, $\mu=(\mu_1,\ldots,\mu_d)\vdash m+n$ and
$\lambda_i\leq\mu_i\leq\lambda_i+1$, $i=1,\ldots,d$.

\begin{corollary}\label{Young rule-special case}
Let $d\geq 2$. If $m\geq n$, then
\begin{equation}\label{Young rule-special case-1}
s_{(m)}(T_d)s_{(n)}(T_d)=\sum_{i=0}^ns_{(m+n-i,i)}(T_d).
\end{equation}
If $n\leq m\leq d$, then
\begin{equation}\label{Young rule-special case-2}
s_{(1^m)}(T_d)s_{(1^n)}(T_d)=s_{(m)'}(T_d)s_{(n)'}(T_d)=\sum_{j=0}^ns_{(m+n-j,j)'}(T_d).
\end{equation}
Here $s_{\mu}(T_d)=0$ if $\mu=(\mu_1,\ldots,\mu_{d'})$ and $\mu_{d+1}>0$.
\end{corollary}

If ${\mathfrak V}_{\text{\rm sup}}$ is a variety of superalgebras and $P_{p,q}({\mathfrak V}_{\text{\rm sup}})$
is the multilinear component of $F({\mathfrak V}_{\text{\rm sup}})$ in the variables $y_1,\ldots,y_p$ and $z_1,\ldots,z_q$,
then we denote by $\chi_{p,q}({\mathfrak V}_{\text{\rm sup}})$ the $S_p\times S_q$-cocharacter of $F({\mathfrak V}_{\text{\rm sup}})$. Then
\begin{equation}\label{cocharacters}
\chi_{p,q}({\mathfrak V}_{\text{\rm sup}})=\sum_{\lambda\vdash p\atop \mu\vdash q}m(\lambda,\mu)\chi_{\lambda}\otimes\chi_{\mu},\,m(\lambda,\mu)\in{\mathbb N}_0.
\end{equation}
Similarly, the Hilbert series $H(F_{p',q'}({\mathfrak V}_{\text{\rm sup}}),U_{p'},V_{q'})$ has the presentation
\begin{equation}\label{Hilbert series and Schur functions}
H(F_{p',q'}({\mathfrak V}_{\text{\rm sup}}),U_{p'},V_{q'})=\sum_{\lambda,\mu}m'(\lambda,\mu)s_{\lambda}(U_{p'})s_{\mu}(V_{q'}),\,m'(\lambda,\mu)\in{\mathbb N}_0
\end{equation}
where $\lambda$ and $\mu$ are partitions in not more
than $p'$ and $q'$ parts, respectively.

A reformulation in the language of super-identities of the following theorem of Giambruno \cite{Giambruno}
gives the relation between the decomposition of the cocharacters and the Hilbert series of varieties of superalgebras.

\begin{theorem}\label{theorem of Giambruno}
For any variety ${\mathfrak V}_{\text{\rm sup}}$ of superalgebras and for $p'\geq p$, $q'\geq q$
the multiplicities $m(\lambda,\mu)$ in {\rm (\ref{cocharacters})} and
$m'(\lambda,\mu)$ in {\rm (\ref{Hilbert series and Schur functions})} are equal.
\end{theorem}

Now we are ready to compute the cocharacter sequence of the variety ${\mathfrak B}_{\text{\rm sup}}$.

\begin{theorem}\label{bicommutative supercocharacters}
The multiplicities $m(\lambda,\mu)$ in the cocharacter sequence of ${\mathfrak B}_{\text{\rm sup}}$ are
\[
m(\lambda,\mu)=\begin{cases}
(\lambda_1-\lambda_2+1)(\mu_1'-\mu_2'+1),\text{ if }\lambda=(\lambda_1,\lambda_2),\mu'=(\mu_1',\mu_2'),\lambda_2+\mu_2'>0,\\
(\lambda_1+1)(\mu_1'+1)-2,\text{ if }\vert\lambda\vert+\vert\mu\vert\geq 2,\lambda_2=\mu_2'=0,\\
1,\text{ if }\vert\lambda\vert+\vert\mu\vert=1,\\
0\text{ otherwise.}
\end{cases}
\]
\end{theorem}

\begin{proof}
The Schur functions $s_{(m)}(U_p)$ and $s_{(1^n)}(V_q)$ are
\[
s_{(m)}(U_p)=\sum u_{i_1}\cdots u_{i_m},\,1\leq i_1\leq\cdots\leq i_m\leq p,
\]
\[
s_{(1^n)}(V_q)=s_{(n)'}(V_q)=\sum v_{j_1}\cdots v_{j_n},\,1\leq j_1<\cdots<j_n\leq q.
\]
Hence (\ref{Hilbert multigraded}) in Theorem \ref{Hilbert series} becomes
\[
H(F_{p,q}({\mathfrak B}_{\text{\rm sup}}),U_p,V_q)=s_{(1)}(U_p)+s_{(1)}(V_q)+\left(\sum_{m\geq 0}s_{(m)}(U_p)\sum_{n=0}^{q}s_{(n)'}(V_q)-1\right)^2.
\]
Let
\[
\left(\sum_{m\geq 0}s_{(m)}(U_p)\sum_{n=0}^qs_{(n)'}(V_q)\right)^2=\sum_{\lambda,\mu}m_1(\lambda,\mu)s_{\lambda}(U_p)s_{\mu}(V_q).
\]
By (\ref{Young rule-special case-1}) and (\ref{Young rule-special case-2}) if $m_1(\lambda,\mu)\not=0$ then
$\lambda=(\lambda_1,\lambda_2)$ and $\mu'=(\mu_1',\mu_2')$ and if $\vert\lambda\vert+\vert\mu\vert\geq 2$, then
\[
m(\lambda,\mu)=\begin{cases}
m_1(\lambda,\mu),\text{ if }\lambda_2>0\text{ or }\mu_2'>0,\\
m_1(\lambda,\mu)-2,\text{ if }\lambda_2=\mu_2'=0.
\end{cases}
\]
By (\ref{Young rule-special case-1}) the Schur function $s_{(\lambda_1,\lambda_2)}(U_p)$ participates in the decomposition of the product
$s_{(m)}(U_p)s_{(n)}(U_p)$ for $m+n=\lambda_1+\lambda_2$ and $m=\lambda_2,\lambda_2+1,\ldots,\lambda_1$, i.e. in $\lambda_1-\lambda_2+1$ products.
Similarly, by (\ref{Young rule-special case-2}) and for sufficiently large $q'$
the Schur function $s_{(\mu_1,\mu_2)'}(V_{q'})$ participates in the decomposition of the product $s_{(m)'}(V_{q'})s_{(n)'}(V_{q'})$
for $\mu\vdash m+n$ and $m=\mu_2',\mu_2+1',\ldots,\mu_1'$, i.e. in $\mu_1'-\mu_2'+1$ products.
Hence
\[
m_1(\lambda,\mu)=(\lambda_1-\lambda_2+1)(\mu_1'-\mu_2'+1)\text{ for }\lambda=(\lambda_1,\lambda_2),\mu'=(\mu_1',\mu_2'),
\]
and this gives the result for $\lambda_2+\mu_2'>0$. For $\lambda=(\lambda_1)$ and
$\mu'=(\mu_1')$ we obtain
\[
m(\lambda,\mu)=(\lambda_1+1)(\mu_1'+1)-2
\]
and this completes the proof.
\end{proof}

For a symmetric function
\[
f(T_d)=\sum_{\lambda}m(\lambda)s_{\lambda}(T_d)
\]
Drensky and Genov \cite{Drensky-Genov} introduced {\it the multiplicity series}
\[
M(f,T_d)=\sum_{\lambda}m(\lambda)T_d^{\lambda}=\sum_{\lambda}m(\lambda)t_1^{\lambda_1}\cdots t_d^{\lambda_d}.
\]
See also \cite{Benanti et al} for applications of the multiplicity series to invariant theory and PI-algebras.
Berele \cite{Berele} considered double Hilbert series instead of Hilbert series
when the number of parts in the partitions with nonzero multiplicities is not bounded.
Combining his idea with the approach in \cite{Drensky-Genov} and \cite{Benanti et al}
we introduce the {\it double multiplicity series} $M({\mathfrak V}_{\text{\rm sup}},U_p,V_q)$
of the variety ${\mathfrak V}_{\text{\rm sup}}$ of superalgebras.
If the multiplicities $m(\lambda,\mu)$ of the cocharacter sequence of ${\mathfrak V}_{\text{\rm sup}}$
are different from 0 for $\lambda=(\lambda_1,\ldots,\lambda_p)$ and
$\mu'=(\mu_1',\ldots,\mu_q')$ only, then
\[
M({\mathfrak V}_{\text{\rm sup}},U_p,V_q)=\sum_{\lambda,\mu}m(\lambda,\mu)u_1^{\lambda_1}\cdots u_p^{\lambda_p}v_1^{\mu_1'}\cdots v_q^{\mu_q'}.
\]

Now we can restate Theorem \ref{bicommutative supercocharacters} in a more compact way.

\begin{corollary}
The double multiplicity series of the variety of bicommutative superalgebras is
\[
\begin{split}
M({\mathfrak B}_{\text{\rm sup}},U_2,V_2)=u_1+v_1&+\frac{1}{(1-u_1)^2(1-u_1u_2)(1-v_1)^2(1-v_1v_2)}\\
&-\frac{2}{(1-u_1)(1-v_1)}+1.
\end{split}
\]
\end{corollary}

\begin{proof}
Since
\[
\frac{1}{(1-u_1)^2(1-u_1u_2)}=\sum_{(\lambda_1,\lambda_2)}(\lambda_1-\lambda_2+1)u_1^{\lambda_1}u_2^{\lambda_2}
\]
we obtain that this is the multiplicity series of the symmetric function
\[
\sum_{(\lambda_1,\lambda_2)}(\lambda_1-\lambda_2+1)s_{(\lambda_1,\lambda_2)}(u_1,u_2).
\]
With similar arguments the expression
\[
\frac{1}{(1-v_1)^2(1-v_1v_2)}=\sum_{(\mu_1',\mu_2')}(\mu_1'-\mu_2'+1)v_1^{\mu_1'}v_2^{\mu_2'}
\]
corresponds to the multiplicity of the $S_q$-character $\chi_{\mu}$ for $\mu'=(\mu_1',\mu_2')\vdash q$.
We complete the proof applying Theorem \ref{bicommutative supercocharacters}.
\end{proof}

\end{document}